\documentclass [12pt,a4paper]{article}

\usepackage{amstext,amsmath,amssymb,amsthm}
\usepackage{ulem}
\usepackage{braket}
\usepackage{amscd}
\usepackage{framed}
\usepackage{fancyhdr}
\usepackage{geometry}
\usepackage[final]{pdfpages}
\usepackage{graphics,color}

\usepackage{tikz}
\usetikzlibrary{
positioning,
3d,
calc,
shapes,
arrows,
intersections,
math,
patterns
}
 \tikzstyle{na} = [baseline=-.5ex]

\usepackage{pgfplots}
\pgfplotsset{compat=1.6}
\pgfplotsset{soldot/.style={color=blue,only marks,mark=*}}
\pgfplotsset{holdot/.style={color=blue,fill=white,only marks,mark=*}}

 \usepackage[colorlinks=true,linkcolor=blue, setpagesize=false, linkbordercolor=blue]{hyperref}
\hypersetup{urlcolor=green, citecolor=green}
 
\usepackage{amsgen}

\usepackage[titletoc]{appendix}

\usepackage{mathrsfs}

\usepackage{wrapfig}
\usepackage{type1cm}

\usepackage{enumerate}

\usepackage{dashbox}

\DeclareMathOperator*{\esssup}{ess\,sup}

\DeclareMathOperator*{\supp}{supp}

\usepackage{esint}

\def\Xint#1{\mathchoice
{\XXint\displaystyle\textstyle{#1}}%
{\XXint\textstyle\scriptstyle{#1}}%
{\XXint\scriptstyle\scriptscriptstyle{#1}}%
{\XXint\scriptscriptstyle\scriptscriptstyle{#1}}%
\!\int}
\def\XXint#1#2#3{{\setbox0=\hbox{$#1{#2#3}{\int}$ }
\vcenter{\hbox{$#2#3$ }}\kern-.6\wd0}}

\def\dashint{\Xint-}

\usepackage{authblk}
\usepackage{blindtext}

\makeatletter
  \def\thefootnote{\ifnum\c@footnote>\z@\leavevmode\lower.5ex%
      \hbox{$^{\@arabic\c@footnote)}$}\fi}
  \makeatother

\allowdisplaybreaks[4]

\newtheorem{thm}{Theorem}[section]
\newtheorem{prop}[thm]{Proposition}
\newtheorem{lem}[thm]{Lemma}
\newtheorem{dfn}[thm]{Definition}
\newtheorem{cor}[thm]{Corollary}
\newtheorem{rmk}[thm]{Remark}

\makeatletter

\@addtoreset{equation}{section}
\makeatother

\begin{document}
\titlepage
\title{\textbf{Regularity estimates for the $p$-Sobolev flow}}
\date{}


\author[$\star$]{Tuomo Kuusi}
\affil[$\star$]{\small{Department of Mathematics and Statistics,
University of Helsinki, PL 68 (Pietari Kalmin katu 5)
00014 University of Helsinki, Finland
 \quad \texttt{tuomo.kuusi@helsinki.fi}}}

  \author[$\ast$]{Masashi Misawa}
 \affil[$\ast$]{\small {Faculty of Advanced Science and Technology, Kumamoto University, Kumamoto 860-8555, Japan\quad
  \texttt{mmisawa@kumamoto-u.ac.jp}}}

  \author[$\dagger$]{Kenta Nakamura}
\affil[$\dagger$]{\small{Headquarters for Admissions and Education, Kumamoto University, Kumamoto 860-8555, Japan
 \quad \texttt{kntkuma21@kumamoto-u.ac.jp}}}

 \maketitle
 \begin{abstract}
We study doubly nonlinear parabolic equation arising from the gradient flow for $p$-Sobolev type inequality, referred as $p$-Sobolev flow from now on,  which includes the classical Yamabe flow on a bounded domain in Euclidean space in the special case $p=2$. In this article we establish a priori estimates and regularity results for the $p$-Sobolev type flow, which are necessary for further analysis and classification of limits as time tends to infinity.
\end{abstract}

\tableofcontents

\section{Introduction}

Let $\Omega \subset \mathbb{R}^{n}\,(n \geq 3)$ be a bounded domain with smooth boundary $\partial \Omega $. For any positive $T \leq \infty$, let $\Omega_T : = \Omega \times (0, T)$ be the space-time cylinder, and let $\partial_p \Omega_T$ be the parabolic boundary defined by $(\partial \Omega \times [0, T))\cup ( \Omega \times \{t = 0\}) $. Throughout the paper we fix  $ p \in [2,n)$ and set $q:=p^\ast - 1$, where $p^\ast := \frac{np}{n-p}$ is the Sobolev conjugate of $p$.
We consider the following doubly nonlinear parabolic equation
\begin{equation}\label{pS}
\begin{cases}
\,\,\partial_t(|u|^{q-1}u)-\Delta_pu=\lambda(t) |u|^{q-1}u\quad &\textrm{in}\quad \Omega_\infty  \\
\,\,u=0\quad &\textrm{on}\quad \!\!\partial\Omega \times (0,\infty) \\
\,\,u(\cdot, 0)=u_0(\cdot ) \quad &\textrm{in}\quad \Omega \\[1mm]
\,\,\displaystyle \| u(t) \|_{L^{q+1}(\Omega)} =1\quad &\textrm{for all}\,\, t \geq 0 .
\end{cases}
\end{equation}
Here the unknown function $u=u(x,t)$ is a real-valued function defined for $(x,t) \in \Omega_\infty$, and the initial data $u_0$ is assumed to be in the Sobolev space $W^{1, p}_0 (\Omega)$, positive, bounded in $\Omega$ and satisfy $ \| u_0 \|_{L^{q+1}(\Omega)}=1$, and $\Delta_p u:=\mathrm{div}\left(|\nabla u|^{p-2}\nabla u\right)$ is the $p$-Laplacian. The condition imposed in the fourth line of \eqref{pS} is called the \textit{volume} constraint and $\lambda (t)$ is Lagrange multiplier stemming from this volume constraint. In fact, multiplying \eqref{pS} by $u$ and integrating by parts, we find by a formal computation that $\displaystyle \lambda (t)=\int_\Omega |\nabla u(x,t)|^p \,dx$ (See Proposition \ref{Energy equality pS} below for the rigorous argument). We call the system \eqref{pS} as \textit{$p$-Sobolev flow}.

\smallskip

Our main result in this paper is the following theorem.
\begin{thm}\label{main theorem}
Let $\Omega$ be a bounded domain with smooth boundary.
Suppose that the initial data $u_0$ is positive in $\Omega$, belongs to $W^{1, p}_0(\Omega) \cap L^\infty (\Omega)$, and satisfies the volume constraint $ \| u_0 \|_{L^{q+1}(\Omega)}=1$.  Let $u$ be a weak solution of \eqref{pS} in $\Omega_\infty\equiv \Omega \times (0, \infty)$ with the initial and boundary data $u_0$. Then, $u$ is positive and bounded in $\Omega_\infty$
and, together with its spatial gradient, locally H\"{o}lder continuous in   $\Omega_\infty$. Moreover,  for $t \in [0,\infty)$,
\begin{equation}\label{te.lambdaineq}
\lambda(t)=\int_{\Omega}|\nabla u(x,t)|^{p}\,dx \qquad \mbox{and} \qquad  \lambda(t) \leq \lambda(0).
\end{equation}

\end{thm}
The definition of a weak solution is given in Definition~\ref{def of weak sol.}.
The global existence of the $p$-Sobolev flow and its asymptotic behavior, that is the volume concentration at infinite time, will be treated in our forthcoming paper,
based on the a-priori regularity estimates for the $p$-Sobolev flow, obtained in the main theorem above.

\smallskip

The ODE part of the $p$-Sobolev flow equation is of exponential type, since the order of solution in both the time derivative and lower-order terms are the same. Thus, the solution is bounded for all times by the maximum principle (Proposition \ref{Boundedness}). On the other hand, {\it a priori} the solution may vanish in a finite time, by the effect of fast diffusion. 
This undesirable behavior can be ruled out for the $p$-Sobolev flow \eqref{pS} by the volume constraint, that is the preservation of volume at all time. In fact, we show the global positivity of solutions of \eqref{pS} under the volume constraint (Proposition \ref{Interior positivity by the volume constraint}). 
The positivity of solutions is based on local energy estimates for truncated solution and De Giorgi's iteration method. For the porous medium and $p$-Laplace equations, see \cite{Vazquez1,Vazquez2,DiBenedetto1,DiBenedetto2}, and also~\cite{Urbano}. Our task is to discover the intrinsic scaling to the doubly nonlinear operator in the $p$-Sobolev flow (Corollary \ref{key cor}). Then, the interior positivity is obtained from some covering argument, being reminiscent of the so-called \textit{Harnack chain} (Corollary \ref{ex. cor prime}). This leads to the positivity and regularity on a \textit{non-convex} domain and thus, may be of interest in geometry. Once the interior positivity is established, the positivity near the boundary of domain follows from the usual comparison function (Proposition \ref{Positivity near the boundary}).
Finally, the H\"{o}lder continuity reduced to that of the evolutionary $p$-Laplacian equation, by use of the positivity and boundedness of solutions. The energy equality also holds true for a weak solution of the $p$-Sobolev flow, leading to the continuity on time of the $p$-energy and volume.

\smallskip

The doubly nonlinear equations have been considered by Vespri \cite{Vespri1}, Porizio and Vespri \cite{Porzio-Vespri}, and Ivanov \cite{Ivanov1, Ivanov2}.  See also \cite{Gianazza-Vespri, Porzio-Vespri, Vespri2,Kinnunen-Kuusi,Kuusi-Siljander-Urbano}.
The regularity proofs for doubly nonlinear equations are based on the intrinsic scaling method, originally introduced by DiBenedetto, and they have to be arranged in some way depending on the particular form of the equation. Here, the very fast diffusive doubly nonlinear equation such as the $p$-Sobolev flow \eqref{pS} is treated, and the positivity, boundedness and regularity of weak solutions are studied and shown in some precise way. See \cite{Nakamura-Misawa} for existence of a weak solution.

\smallskip

Consider next the stationary equation for \eqref{pS}, which is described by the $p$-Laplacian type elliptic equation, obtained from simply removing the time derivative term from the first equation in \eqref{pS}. This stationary equation relates to the existence of the extremal function attaining the best constant of Sobolev's embedding inequality, $W^{1,p}_0 (\Omega) \hookrightarrow L^{q+1} (\Omega)$.
If the domain $\Omega$ is star-shaped with respect to the origin, the trivial solution $u \equiv 0$ only exists, by Pohozaev's identity and Hopf's maximum principle. Thus, the $p$-Sobolev flow \eqref{pS}, if globally exists, may have finitely many volume concentration points at infinite time. This volume concentration phenomenon is one of our motives of studying the $p$-Sobolev flow \eqref{pS}. Moreover, if the domain $\Omega$ is replaced by a smooth compact manifold, we can study the generalization of Yamabe problem in the sense of $p$-Laplacian setting. This is another geometric motive for the $p$-Sobolev flow.

%
%

\smallskip

In fact, in the case that $p=2$, our $p$-Sobolev flow (\ref{pS}) is exactly the classical Yamabe flow equation in the Euclidean space. The classical Yamabe flow was originally introduced by Hamilton in his study of the so-called Yamabe problem (\cite{Yamabe, Aubin1, Aubin2}), asking the existence of a conformal metric of constant curvature on $n(\geq 3)$-dimensional closed Riemannian manifolds (\cite{Hamilton89}). Let $(\mathcal{M},g_{0})$ be a $n(\geq 3)$-dimensional smooth, closed Riemannian manifold with scalar curvature $R_{0}=R_{g_{0}}$. The classical Yamabe flow is given by the heat flow equation
\begin{equation}
\label{heatflow}
u_t=(s-R)u=u^{-\frac{4}{n-2}}(c_n \Delta_{g_0}u-R_0u)+su,
\end{equation}
where $u=u(t), t\geq 0$ is a positive function on $\mathcal{M}$ such that $g(t)=u(t)^{\frac{4}{n-2}}g_0$ is a conformal change of a Riemannian metric $g_{0}$, with volume  $\displaystyle \mathrm{Vol}(\mathcal{M})=\int_\mathcal{M}\,dvol_g=\int_\mathcal{M} u^{\frac{2n}{n-2}}dvol_{g_0}=1$, having total curvature
$$\displaystyle s:=\int_\mathcal{M}(c_n |\nabla u|_{g_0}^2+R_0u^2)\,dvol_{g_0}=\int_\mathcal{M}R\,dvol_{g},\quad c_n:=\frac{4(n-1)}{n-2}.$$
Here we will notice that the condition for volume above naturally corresponds the volume constraint in \eqref{pS}.
Hamilton (\cite{Hamilton89}) proved convergence of the Yamabe flow as $t \to \infty$ under some geometric conditions. Under the assumption that $(\mathcal{M},g_0)$ is of positive scalar curvature and locally conformal flat, Ye (\cite{Ye}) showed the global existence of the Yamabe flow and its convergence as $t \to \infty$ to a metric of constant scalar curvature. Schwetlick and Struwe (\cite{Schwetlick-Struwe})
succeeded in obtaining the asymptotic convergence of the Yamabe flow in the case $3 \leq n \leq 5$, under an appropriate condition of Yamabe invariance $Y(\mathcal{M},g_0)$, which is given by infimum of Yamabe energy $E(u)=\int_\mathcal{M}(c_n |\nabla u|_{g_0}^2+R_0u^2)\,dvol_{g_0}$ among all positive smooth function $u$ on $\mathcal{M}$ with $\mathrm{Vol}(\mathcal{M})=1$, for an initial positive scalar curvature. In Euclidean case, since $R_{g_{0}}=0$ their curvature assumptions are not verified. In outstanding results concerning the Yamabe flow, the equation is equivalently transformed to the scalar curvature equation, and this is crucial for obtaining many properties for the Yamabe flow.
In this paper, we are forced to take a direct approach dictated by the structure of the $p$-Laplacian leading to the degenerate/singular parabolic equation of the $p$-Sobolev flow. Let us remark that our results cover those of the classical Yamabe flow in the Euclidian setting.

%

\medskip

The structure of this paper is as follows. In Section \ref{Sec. Preliminaries} we prepare some notations and technical analysis tools, which are used later.
Section \ref{Sec. Doubly nonlinear equations of p-Sobolev flow type} provides basic definitions of weak solutions, and also some basic study the doubly nonlinear equations of $p$-Sobolev flow type, including (\ref{pS}), and derivation of the minimum and maximum principles. Moreover, we establish the comparison theorem and make the Caccioppoli type estimates, which have a crucial role in Section \ref{Expansion of positivity section}.
In the next section, Section \ref{Expansion of positivity section}, we prove the expansion of positivity for the doubly nonlinear equations of $p$-Sobolev flow type. In 
Section \ref{Sect. The p-Sobolev flow} we show the positivity, the energy estimates and, consequently, the H\"{o}lder regularity for the p-Sobolev flow (\ref{pS}). Finally, in Appendix \ref{Sec. Some fundamental facts}, \ref{Sec. Proof energy equality} and \ref{Sec. Notes on Holder regularity}, for the $p$-Sobolev flow, we give detailed proofs of facts used in the previous sections.
%

 \section{Preliminaries}\label{Sec. Preliminaries}
 We prepare some notations and technical analysis tools, which are used later.
 \subsection{Notation}

 Let $\Omega$ be a bounded domain in $\mathbb{R}^n\,\,(n \geq 3)$ with smooth boundary $\partial \Omega$ and for a positive $T \leq \infty$ let $\Omega_T:=\Omega \times (0,T)$ be the cylindrical domain. Let us define the parabolic boundary of $\Omega_T$ by
\begin{equation*}
\partial_p \Omega_T:= \partial \Omega \times [0, T)\cup \Omega \times \{t = 0\}.
 \end{equation*}

We prepare some function spaces, defined on space-time region.  For $1 \leq p,q \leq \infty$, $L^{q}(t_1,t_2\,;\,L^{p}(\Omega))$ is a function space of measurable real-valued functions on a space-time region $\Omega \times (t_1,t_2)$ with a finite norm
\begin{equation*}
\|v\|_{L^{q}(t_1,t_2\,;\,L^{p}(\Omega))}:=
\begin{cases}
\displaystyle \left(\int_{t_1}^{t_2}\|v(t)\|_{L^{p}(\Omega)}^{q}\,dt \right)^{1/q}\quad &(1 \leq q<\infty) \\
\displaystyle \esssup_{t_1 \leq t \leq t_2}\|v(t)\|_{L^{p}(\Omega)}\quad &(q=\infty),
\end{cases}\notag
\end{equation*}
where
\begin{equation*}
\|v(t)\|_{L^{p}(\Omega)}:=
\begin{cases}
\left(\displaystyle \int_{\Omega}|v(x,t)|^{p}\,dx \right)^{1/p}\quad &(1 \leq p <\infty )\\
\esssup \limits_{x \in \Omega}|v(x,t)|\quad &(p=\infty).
\end{cases}
\end{equation*}
When $p=q$, we write $L^p(\Omega \times (t_1,t_2))=L^{p}(t_1,t_2\,;\,L^{p}(\Omega))$ for brevity. For $1 \leq p <\infty$ the Sobolev space $W^{1,p}(\Omega)$ is consists of measurable real-valued functions that are weakly differentiable and their weak derivatives are $p$-th integrable on $\Omega$, with the norm
\begin{equation*}
\|v\|_{W^{1,p}(\Omega)}:=\left(\int_{\Omega}|v|^{p}+|\nabla v|^{p}\,dx \right)^{1/p},\notag
\end{equation*}
where  $\nabla v=(v_{x_1},\ldots, v_{x_n})$  denotes the gradient of $v$ in a distribution sense, and let $W_{0}^{1,p}(\Omega)$ be the closure of $C_{0}^{\infty}(\Omega)$ with resptect to the norm $\|\cdot\|_{W^{1,p}}$.  Also let $L^{q}(t_1,t_2\,;\,W_{0}^{1,p}(\Omega))$ denote a function space of measurable real-valued functions on space-time region with a finite norm
\begin{equation*}
\|v\|_{L^{q}(t_1,t_2\,;\,W_{0}^{1,p}(\Omega))}:=\left(\int_{t_1}^{t_2}\|v(t)\|_{W^{1,p}(\Omega)}^{q} \,dt\right)^{1/q}.
\end{equation*}

Let $B=B_\rho(x_0):=\{x \in \mathbb{R}^n\,:\,|x-x_0|<\rho\}$ denote an open ball with radius $\rho>0$ centered at some $x_0 \in \mathbb{R}^n$. Let $E \subset \mathbb{R}^{n}$ be a bounded domain. For a real number $k$ and for a function $v$ in $L^{1}(E)$ we define the \textit{truncation} of $v$ by
\begin{equation}\label{truncation}
(v-k)_+:=\max\{(v-k),0\};\quad (k-v)_+:=\max\{(k-v),0\}.
\end{equation}

For a measurable function $v$  in $L^{1}(E)$ and a pair of real numbers $k<l$, we set
\begin{equation}\label{def of sets}
\begin{cases}
E \cap \{v>l\}:=\{x \in E\,:\,v(x)>l\} \\
E \cap \{v<k\}:=\{x \in E\,:\, v(x)<k\} \\
E \cap \{k<v<l\}:=\{x \in E\,:\,k<v(x)<l\}.
\end{cases}
\end{equation}

Let $z=(x,t) \in \mathbb{R}^{n}\times \mathbb{R}$ be a space-time variable and  $dz=dxdt$ be the space-time volume element.
%
%

\subsection{Technical tools}

We first recall the following De Giorgi's inequality (see \cite{DiBenedetto1}).

\begin{prop}[De Giorgi's inequality] \label{De Giorgi's inequality}
Let $v \in W^{1,1}(B)$ and $k,l \in \mathbb{R}$ satisfying $k<l$. Then there exists a positive constant $C$ depending only on $p,n$ such that
\begin{equation}\label{De Giorgi's ineq.}
(l-k)\big|B \cap \{v>l\}\big| \leq C \frac{\rho^{n+1}}{\big|B \cap \{v <k\} \big|}\int_{B \cap \{k <v<l\}}|\nabla v|\,dx.
\end{equation}
\end{prop}

Let $q=np/(n-p)-1$ as before. Following \cite{DiBenedetto1}, we define the auxiliary function \begin{equation}
\begin{cases}\label{def of A}
 A^{+}(k,u):=\displaystyle\int_{k^q}^{u^q}\left(\xi^{1/q}-k\right)_+\,d\xi \\[3mm]
 A^{-}(k,u):=\displaystyle\int_{u^q}^{k^q}\left(k-\xi^{1/q}\right)_+\,d\xi
 \end{cases}
\end{equation}
for $u \geq 0$ and $k \geq 0$. Changing a variable $\eta=\xi^{1/q}$, we have
\begin{equation*}
A^{+}(k,u)= q\int_k^u(\eta-k)_+\eta^{q-1}\,d\eta =q \int_0^{(u-k)_+}(\eta+k)^{q-1}\eta\,d\eta
\end{equation*}
and \begin{equation*}
A^{-}(k,u)= q\int_u^k(k-\eta)_+\eta^{q-1}\,d\eta =q \int_0^{(k-u)_+}(k-\eta)^{q-1}\eta\,d\eta.
\end{equation*}Then we formally get
\begin{equation}\label{eq. of A^{+}}
\frac{\partial }{\partial t}A^{+}(k,u)=\frac{\partial u^q}{\partial t}(u-k)_+
\end{equation}
and
\begin{equation}\label{eq. of A^{-}}
\frac{\partial }{\partial t}A^{-}(k,u)=-\frac{\partial u^q}{\partial t}(k-u)_+.
\end{equation}
If $k = 0$, we abbreviate as
\begin{equation*}
A^+ (u) = A^+ (0, u) \quad \textrm{and} \quad A^- (u) = A^- (0, u).
\end{equation*}
 Let 
 $0 < t_1 < t_{2} \leq T$ and let $K$ be any domain in $\Omega$.
We denote a parabolic cylinder by $K_{t_1, t_{2}} := K \times (t_1, t_{2})$. We recall the Sobolev embedding of parabolic type.
\begin{prop}[\cite{DiBenedetto1}]\label{parabolic embedding1}
There exists a constant $C$ depending only on $n,p,r$ such that for every $v \in L^\infty(t_1, t_2 ; L^r (K)) \cap L^p (t_1, t_2 ; W_0^{1, p} (K))$
\begin{align}\label{parabolic embedding1 ineq}
\int_{K_{t_1, t_{2}}}|v|^{p\frac{n+r}{n}}\,dz \notag \leq C \left(\int_{K_{t_1, t_{2}}} |\nabla v|^p\,dz \right) \left(\esssup_{t_{1}<t<t_{2}}\int_\Omega |v|^r\,dx \right)^\frac{p}{n}
\end{align}
\end{prop}

We next use so-called \textit{fast geometric convergence} which will be employed later on many times. See \cite{DiBenedetto1} for details.

\begin{lem}[Fast geometric convergence, \cite{DiBenedetto1}]\label{Fast geometric convergence}
Let $\{Y_m\}_{m=0}^\infty$ be a sequence of positive numbers, satisfying the recursive inequlities
\begin{equation}
Y_{m+1} \leq Cb^mY_m^{1+\alpha},\quad m=0,1,\ldots,
\end{equation}
where $C,b>1$ and $\alpha>0$ are given constants independent of $m$. If the initial value $Y_0$ satisfies
\begin{equation}
Y_0 \leq C^{-1/\alpha}b^{-1/\alpha^2},
\end{equation}
then $\lim \limits_{m \to \infty}Y_m=0$.
\end{lem}

We also need a fundamental algebraic inequality, associated with the $p$ -Laplace operator (see \cite{Barrett-Liu}).
\begin{lem}\label{Barrett-Liu}
For all $p \in (1,\infty)$ there exist positive constants $C_{1}(p,n)$ and $C_{2}(p,n)$ such that for all $\xi,\,\eta \in \mathbb{R}^{n}$
\begin{equation}\label{BL1}
||\xi|^{p-2}\xi-|\eta|^{p-2}\eta| \leq C_{1}(|\xi|+|\eta|)^{p-2}|\xi-\eta|
\end{equation}
and
\begin{equation}\label{BL2pre}
(|\xi|^{p-2}\xi-|\eta|^{p-2}\eta)\cdot (\xi-\eta) \geq C_2 (|\xi|+|\eta|)^{p-2}|\xi-\eta|^2,
\end{equation}
where dot $\cdot $ denotes the inner product in $\mathbb{R}^{n}$.  In particular, if $p\geq 2$, then
\begin{equation}\label{BL2}
(|\xi|^{p-2}\xi-|\eta|^{p-2}\eta)\cdot (\xi-\eta) \geq C_2 |\xi-\eta|^p.
\end{equation}
\end{lem}


\section{Doubly nonlinear equations of $p$-Sobolev flow type}\label{Sec. Doubly nonlinear equations of p-Sobolev flow type}
 Let $T \leq \infty$. 
 We study the following a doubly nonlinear equation of \textit{$p$-Sobolev flow type}:
\begin{equation}\label{maineq}
\begin{cases}
\partial_{t}(|u|^{q-1}u)-\Delta_pu=c|u|^{q-1}u \quad \mathrm{in}\quad \Omega_T\\
 0\leq u \leq M \quad \mathrm{on}\quad \partial_{p}\Omega_{T},
\end{cases}
\end{equation}
where $u=u(x,t):\Omega_{T} \to \mathbb{R}$ be unknown real valued function, and $c$ and $M$ are nonnegative constant and positive one, respectively. This section is devoted to some a priori estimates of a weak solution to (\ref{maineq}).
Firstly, we recall the definition of weak solution of (\ref{maineq}).
\begin{dfn}\label{def}\normalfont
 A measurable function $u$ defined on $\Omega_{T}$ is called a \textit{weak supersolution (subsolution)} to (\ref{maineq}) if the following (D1)-(D3) are satisfied.
 \begin{enumerate}[(D1)]
 \item $u \in L^{\infty}(0,T\,;\,W^{1,p}(\Omega))$; \,\, $\partial_{t}(|u|^{q-1}u) \in L^{2}(\Omega_{T})$
 \item For every nonnegative $\varphi \in C^{\infty}_{0}(\Omega_T) $,
 \begin{equation*}-\int_{\Omega_T}|u|^{q-1}u\varphi_{t}\,dz+\int_{\Omega_T}|\nabla u|^{p-2}\nabla u\cdot \nabla \varphi\,dz \geq (\leq) c\int_{\Omega_T} |u|^{q-1}u\varphi dz.
 \end{equation*}
 \item $0 \leq u \leq M $ on $\partial_p \Omega_T$ in the trace sense: \\
$(-u(t))_+,\,( (u(t))_+-M)_+ \in W^{1, p}_0 (\Omega)$ for almost every $t \in (0, T)$;
\begin{equation*}\int_{\Omega}(-u(x,t))_+^{q+1}dx,  \quad \int_{\Omega}\left( (u(x,t))_+-M\right)_+^{q+1} \,dx \rightarrow 0 \quad \textrm{as}\quad  t\searrow 0.
\end{equation*}
 \end{enumerate}
A measurable function $u$ defined on $\Omega \times [0,T]$ is called a \textit{weak solution} to (\ref{maineq}) if it is simultaneously a weak sub and supersolution; that is,
\begin{equation*}-\int_{\Omega_T}|u|^{q-1}u\varphi_{t}\,dz+\int_{\Omega_T}|\nabla u|^{p-2}\nabla u\cdot \nabla \varphi\,dz = c\int_{\Omega_T} |u|^{q-1}u\varphi dz
 \end{equation*}
 for every $\varphi \in C^{\infty}_{0}(\Omega_T). $
 \end{dfn}

Similarly,

\begin{dfn}\label{def of weak sol.}\normalfont
A measurable function $u$ defined on $\Omega_{T}$ is called a \textit{weak solution} of \eqref{pS}
 if the following (D1)-(D4) are satisfied.
 \begin{enumerate}[(D1)]
 \item $u \in L^{\infty}(0,T\,;\,W^{1,p}(\Omega))$; \,\, $\partial_{t}(|u|^{q-1}u) \in L^{2}(\Omega_{T})$
 \item There exists a function $\lambda (t) \in L^1 (0, T)$ such that, for every $\varphi \in C^{\infty}_{0}(\Omega_T) $,
 \begin{equation*}-\int_{\Omega_T}|u|^{q-1}u\varphi_{t}\,dz+\int_{\Omega_T}|\nabla u|^{p-2}\nabla u\cdot \nabla \varphi\,dz = \int_{\Omega_T} \lambda(t) |u|^{q-1}u\varphi dz.
 \end{equation*}
 \item $\|u (t)\|_{L^{q +1} (\Omega)} = 1$ for all $t \ge 0$.
 \item $u(0) = u_0$ in $\Omega$ and $u = 0 $ on $\partial_p \Omega \times (0,T)$ in the trace sense: \\
 $u(t) \in W^{1, p}_0 (\Omega)$ for almost every $t \in (0, T)$;
\[
\|u (t) - u_0\|_{L^{q+1} \Omega)}\rightarrow 0 \quad \textrm{as}\quad  t\searrow 0.
\]
%
%
 \end{enumerate}
 \end{dfn}

\begin{rmk}\normalfont
A solution of our $p$-Sobolev flow equation (\ref{pS}) is a subsolution of (\ref{maineq}) with $c=\|\nabla u_0\|_{L^p(\Omega)}^p$ and a supersolution of (\ref{maineq})
with $c=0$, respectively. See the energy estimate \eqref{lambdaineq} in Proposition \ref{Energy equality pS} below.
\end{rmk}

\subsection{Nonnegativity and boundedness}

We next claim that a weak supersolutions to (\ref{maineq}) are nonnegative, i.e., they satisfy the weak minimum principle.
\begin{prop}[Nonnegativity]\label{Nonnegativity prop}
A weak supersolution $u$ to (\ref{maineq}) satisfies
\begin{equation}\label{nonnegativity}
u \geq 0 \quad \textrm{in}\quad \Omega_T.
\end{equation}
\end{prop}

\begin{proof}
If $u$ is a weak supersolution to (\ref{maineq}), $-u$ is a weak subsolution. We note by (D1) in Definition \ref{def} that
$\partial_t (|u|^{q-1}u) \in L^2(\Omega_T)$ and that $(-u)_+ \in L^\infty (0,T\,; W^{1,p}_0 (\Omega)) \subset L^{q + 1} (\Omega_T) \subset L^2 (\Omega_T)$. Let  $0 < t_1 < t \le T$ be arbitrarily taken and fixed.
Put $\Omega_{t_1, t}=\Omega \times (t_1, t)$. Let $\delta$ be any positive number
such that $\delta \le (t - t_1)/3$. We define a Lipschitz cut-off function on time, $\sigma_{t_1, t}$
such that
\begin{equation*}
0 \leq \sigma_{t_1,t} \leq 1,\quad \sigma_{t_{1}, t} = 1
\quad\textrm{in}\quad  (t_1+\delta, t-\delta)
\quad
\textrm{and}
\quad
\supp(\sigma_{t_1, t}) \subset (t_1, t).
\end{equation*}
\begin{center}
\begin{tikzpicture}[domain=-1:5, samples=70, thick,scale=1.3]
\draw[thin, ->] (-1.1,0)--(2.3,0) node[right] {\footnotesize $s$}; 
\draw[thin, ->] (-0.2,-0.5)--(-0.2,1.5) node[above] {\footnotesize $y$}; 
 \draw (0.3,0)node[below]{\footnotesize $t_{1}+\delta$};
\filldraw(0.3,0) circle (0.03);
\filldraw(1.7,0) circle (0.03);
 \draw (1.7,0)node[below]{\footnotesize $t-\delta$};
\draw[thin, densely dotted] (0.3,0) -- (0.3,1.2);
\draw[thin, densely dotted] (1.7,0) -- (1.7,1.2);
\draw[thin, densely dotted] (-0.2,1.2) -- (0.3,1.2);
\draw (-0.3,1.2)node {\footnotesize$1$};
\draw [domain=0:0.3]plot(\x, {4 * \x  }); 
\draw [domain=1.7:2.0]plot(\x, {-4 * \x +8  }); 
 \draw (0.3,1.2)--(1.7,1.2) node [above] {$\sigma_{t_{1},t}$}; 
\end{tikzpicture}
\end{center}
\noindent
A function $(-u)_+\sigma_{t_1, t}$ is an admissible test function in (D2) of Definition \ref{def}. From (D2) of Definition \ref{def} we obtain that
\begin{align}\label{eq1 of nonnegativity}
\int_{\Omega_{t_1, t}}\partial_{t}(|u|^{q-1}(-u))(-u)_{+}\sigma_{t_1, t}\,dz&+\int_{\Omega_{t_1, t}}|\nabla u|^{p-2}\nabla (-u)\cdot \nabla\left((-u)_{+}\sigma_{t_1, t} \right)\,dz \notag \\
&\leq c\int_{\Omega_{t_1, t}}|u|^{q-1}(-u)(-u)_{+}\sigma_{t_1, t}\,dz.
\end{align}
The first integral on the left hand side of (\ref{eq1 of nonnegativity}) is computed as
\begin{align}\label{eq2 of nonnegativity}
\int_{\Omega_{t_1, t}}\partial_{t}(|u|^{q-1}(-u))(-u)_{+}\sigma_{t_1, t}\,dz
&=\int_{\Omega_{t_1, t}}\partial_t A^+ ((- u)_+)\sigma_{t_1, t}dz \notag \\
&=\int_{\Omega}A^+((-u)_+\sigma_{t_1, t} )\,dx \bigg|_{t_1}^t-\int_{\Omega_{t_1, t}}A^+((-u)_+) \partial_t \sigma_{t_1, t}dz \notag \\
&=\frac{1}{\delta}\left(\int_{t -\delta}^t-\int_{t_1}^{t_1 + \delta}\right)\int_\Omega A^+((-u (s))_+)dxds
\end{align}
which, as $\delta \to 0$, converges to
\begin{multline}\label{eq2' of nonnegativity}\int_\Omega
A^+ ((-u(t))_+)dx-\int_\Omega A^+ ((-u(t_1))_+)dx
\\
=\frac{q}{q+1}
\int_\Omega(-u(t))_+^{q+1}dx-\frac{q}{q+1}
\int_\Omega (-u(t_1))_+^{q+1}dx,
\end{multline}
where by (D3) in Definition \ref{def} we have that
\begin{equation}\label{eq2'' of nonnegativity}
\displaystyle \lim_{t_1 \rightarrow 0}
\int_\Omega(-u (t_1))_+^{q+1}dx
\rightarrow 0\quad \textrm{as}\quad  t_1 \rightarrow 0.
\end{equation}
The second integral on the left hand side of (\ref{eq1 of nonnegativity}) is bounded from below as
\begin{equation}\label{eq3 of nonnegativity}
 \int_{\Omega_{t}}|\nabla (-u)_{+}|^{p}\sigma_{t_1, t}\,dz \geq 0.
\end{equation}
Taking the limit in (\ref{eq1 of nonnegativity}) as $\delta \searrow 0$ and $t_1 \searrow 0$, and combining (\ref{eq2 of nonnegativity}), (\ref{eq2' of nonnegativity}), (\ref{eq2'' of nonnegativity}) with (\ref{eq3 of nonnegativity}), we get
\begin{equation*}
\frac{q}{q+1}\int_{\Omega}(-u(t))_{+}^{q+1}\,dx \leq c\int_0^t\int_{\Omega}(-u(\tau))_+^{q+1}\,dxd\tau
\end{equation*}
and, by Gronwall's lemma,
\begin{equation*}
 \int_{\Omega}(-u(t))_{+}^{q+1}\,dx \leq 0
\end{equation*}
since again, by (D3) of Definition \ref{def}, $(-u(t))_+ \to 0$ in $L^{q+1} (\Omega)$
as $t \searrow 0$. Thus we have $-u(x,t) \leq 0$ \,for  $(x,t) \in \Omega_{T}$ and the claim is verified.
\end{proof}
We next show the boundedness of the solution.
\begin{prop}[Boundedness]\label{Boundedness}
Let $u$ be a weak subsolution of (\ref{maineq}) such that
$(u(t))_+ \in W^{1, p}_0 (\Omega)$ for almost every $t \in (0, T)$. Then
\begin{equation}\label{boundedness}
\|(u(t))_+ \|_{L^\infty (\Omega)} \leq e^{c T/q}\|u_{0}\|_{L^{\infty}(\Omega)}.
\end{equation}

\end{prop}

\begin{proof}
We will follow the similar argument as in \cite{Alt-Luckhaus}. Set $M:=\|u_{0}\|_{L^{\infty}(\Omega)}$, so that $0 \leq u \leq M$ on $\partial_p \Omega_T$.  Let us define, for a small $\delta>0$, the Lipschitz truncated function~$\phi_{\delta}(u)$ by
\begin{equation*}
\phi_{\delta}(u):=\min\left\{1,\,\frac{(e^{-ct/q}u-M)_{+}}{\delta}\right\},
\end{equation*}
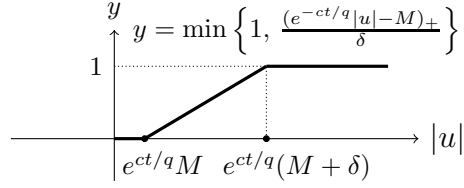
\begin{center}
\begin{tikzpicture}[domain=-2:2, samples=70, very thick,scale=0.8]
\draw[thin, ->] (-2.2,0)--(4.5,0) node[right] {$|u|$}; 
\draw[thin, ->] (-0.5,-0.7)--(-0.5,1.8) node[above] {\footnotesize$y$}; 
\draw [domain=0:2]plot(\x, {0.6 *\x}); 
\draw (2.5,1.2)node[above] {\footnotesize $y=\min\left\{1,\,\frac{(e^{-ct/q}|u|-M)_{+}}{\delta}\right\}$};
\draw (0.3,0)node[below]{\footnotesize $e^{ct/q}M$};
\filldraw(0,0) circle (0.03);
\filldraw[very thick](-0.5, 0)--(0,0);
\filldraw[very thick](2, 1.2)--(4,1.2);
\filldraw(2,0) circle (0.03);
\draw (2.5,0)node[below]{\footnotesize $e^{ct/q}(M+\delta)$};
\draw[thin, densely dotted] (2,0) -- (2,1.2);
\draw[thin, densely dotted] (-0.5,1.2) -- (4,1.2);
\draw (-0.8,1.2)node {\footnotesize$1$};
\end{tikzpicture}
\end{center}
where we note that the support of $\phi_\delta $ is $\{u>e^{ct/q}M\}$, and $\phi_\delta (u),\phi_\delta' (u) \in L^\infty (\Omega_T)$, $\phi_\delta(0)=0$,  and further $\phi_\delta (u)
\in L^\infty (0,T\,; W^{1,p}_0 (\Omega))$. Let $0<t_1<t\leq T$
and $\sigma_{t_1,t}$ be the same time cut-off function as in the proof of Proposition \ref{Nonnegativity prop}. The function $e^{-ct}\sigma_{t_1, t} \phi_\delta (u) $ is an admissible test function in (D2) in Definition \ref{def}.
Choose a test function as $e^{-ct}\sigma_{t_1, t}  \phi_\delta (u)$ in (D2) in Definition \ref{def} to have
\begin{equation}\label{eq1. of boundedness}
\int_{\Omega_{t_{1},t}}\partial_{t}(e^{-ct}|u|^{q-1} u)\sigma_{t_1, t}  \phi_\delta (u)\,dz+\int_{\Omega_{t_{1},t}}|\nabla u|^{p-2}\nabla u\cdot \nabla\left(e^{-ct}\sigma_{t_1, t}  \phi_\delta (u)\right)\,dz \leq 0.
\end{equation}
The first term on the left of (\ref{eq1. of boundedness}) is computed as
\begin{equation}\label{eq2. of boundedness}
\int_{\Omega_{t_{1},t}}\partial_{t}(e^{-ct} |u|^{q-1} u )\min\left\{1,\,\frac{(e^{-ct/q} u -M)_{+}}{\delta}\right\}\sigma_{t_1, t}\,dz.
\end{equation}
Since, on the support of $\phi_\delta, \{u > e^{ct/q}M\}$,
\begin{align*}
\nabla u \cdot \nabla \phi_{\delta}(u)=
\frac{1}{\delta}\chi_{\{e^{ct/q}M< u \leq e^{ct/q}(M+\delta)\}}|\nabla u|^{2},
\end{align*}
the second term  is estimated  as
\begin{align}\label{eq3. of boundedness}
%
\frac{|\nabla u|^{p}}{\delta}\chi_{\{e^{ct/q}M< u \leq e^{ct/q}(M+\delta)\}}\sigma_{t_1, t}e^{-ct} \,dz \geq 0.
\end{align}
%
Gathering (\ref{eq1. of boundedness}), (\ref{eq2. of boundedness}) and (\ref{eq3. of boundedness}), we obtain
\begin{equation}\label{eq4. of boundedness}
\int_{\Omega_{t_1, t}}\partial_{t}(e^{-ct}|u|^{q-1} u)\min\left\{1,\,\frac{(e^{-ct/q} u -M)_{+}}{\delta}\right\}\sigma_{t_1, t} \,dz \leq 0.
\end{equation}
Since  $\partial_{t}(|u|^{q-1}u)=\partial_{t} u^{q} \in L^{2}(\Omega)$ in $\{(u)_+>0\}$ by (D1) of Definition \ref{def}, it holds that $\partial_t (e^{-ct} (u)_+^q) \in L^{2}(\Omega_{T})$. Taking the limit as $\delta \searrow 0$ in (\ref{eq4. of boundedness}), by the Lebesgue dominated convergence theorem, we have that
\begin{equation*}
\int_{\Omega_{t_{1},t}}\partial_{t}(e^{-ct} u^{q})\chi_{\{ u >e^{ct/q}M\}}\,dz \leq 0,
\end{equation*}
namely,
\begin{equation}\label{eq5. of boundedness}
\int_{\Omega_{t_{1},t}}\partial_{t}(e^{-ct} u ^{q}-M^{q})_{+}dxdt \leq 0.
\end{equation}
%
%
%
By (D3) in Definition \ref{def}
\begin{equation*}
\int_{\Omega}(e^{-ct_{1}} (u(t_{1}))_+^{q}-M^{q})_{+}dx \leq \int_{\Omega}((u(t_{1}))_+^{q}-M^{q})_{+}dx\to 0
\end{equation*}
%
%
%
as $t_1 \searrow 0$. Hence, we pass to the limit as $t_1 \searrow 0$  in (\ref{eq5. of boundedness}) to have
\begin{equation*}
\int_{\Omega}(e^{-ct} (u(t))^{q}-M^{q})_{+}dx\leq 0
\end{equation*}
if and only if $(u)_+\leq e^{ct/q}M$ \,in $\Omega  \times [0,T]$, and we arrive at the assertion.
\end{proof}

\subsection{Comparison theorem}

We recall the crucial fact, addressed by Alt-Luckhaus (\cite{Alt-Luckhaus}), the \textit{Comparison theorem} (\cite[Theorem 2.2, p.325]{Alt-Luckhaus}).  For stating it without loss of generality, we define a weak sub and super solutions. A measurable function $u$ and $v$ on $\Omega_T$ are a weak supersolution and subsolution, respectively, if the conditions (D1) and (D2) in Definition \ref{def} are satisfied.
We say that $u \geq v$ on $\partial_p \Omega_T$ in the trace sense, if
\begin{enumerate}[(D3')]
 \item $(-u (t)+v (t))_+ \in W^{1, p}_0 (\Omega)$, for almost every $t \in (0,T)$, and \\[2mm]
 $(- |u|^{q - 1} u (t)
+ |v|^{q - 1} v (t))_+ \to 0$ in $L^1 (\Omega)$ as $t \searrow 0$.
\end{enumerate}
\begin{thm}[Comparison theorem, \cite{Alt-Luckhaus}]\label{Comparison theorem}
Let $u$ and $v$ be a weak supersolution and subsolution to (\ref{maineq}) in $\Omega_T$, respectively. If $u \geq v$ in the sense of (D3')  on $\partial_p \Omega_T$, then it holds true that\begin{equation*}
u \geq v\quad in \quad \Omega_T.
\end{equation*}
\end{thm}

\begin{proof}As before, for a small $\delta >0$, let us define the Lipschitz  function $\phi_{\delta}$ by %
\begin{equation*}
\phi_{\delta}(x):=\min\left\{1,\frac{x_{+}}{\delta} \right\}.
\end{equation*}%
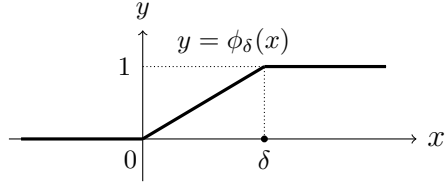
\begin{center}
\begin{tikzpicture}[domain=-2:2, samples=70, very thick,scale=0.8]
\draw[thin, ->] (-2.2,0)--(4.5,0) node[right] {$x$}; 
\draw[thin, ->] (0,-0.7)--(0,1.8) node[above] {\footnotesize$y$}; 
\draw [domain=0:2]plot(\x, {0.6 *\x}); 
\draw (1.5,1.2)node[above] {\footnotesize $y=\phi_\delta(x)$};
\draw (-0.2,0)node[below]{\footnotesize 0};
\filldraw[very thick](-2, 0)--(0,0);
\filldraw[very thick](2, 1.2)--(4,1.2);
\filldraw(2,0) circle (0.03);
\draw (2,0)node[below]{\footnotesize $\delta$};
\draw[thin, densely dotted] (2,0) -- (2,1.2);
\draw[thin, densely dotted] (0,1.2) -- (4,1.2);
\draw (-0.3,1.2)node {\footnotesize$1$};
\end{tikzpicture}
\end{center}
Note that $\phi_\delta (v-u) \in L^\infty (\Omega_T)$ and $L^\infty(0, T; W^{1,p}_0 (\Omega))$. Let $0 < t_1 < t \le T$
and $\sigma_{t_1,t}$ be the same time cut-off function as in the proof of Proposition \ref{Nonnegativity prop}. Choose a test function $\sigma_{t_{1}, t} \phi_\delta (v-u)$, which is admissible, to have
\begin{align}\label{eq1. of comparison theorem}
\int_{\Omega_{t_{1},t}}\partial_{t}(|u|^{q-1}u)\phi_{\delta}(v-u)\sigma_{t_{1}, t}\,dz&+\int_{\Omega_{t_{1},t}}|\nabla u|^{p-2}\nabla u \cdot \nabla(\phi_{\delta}(v-u))\sigma_{t_{1}, t}\,dz \notag \\
&\quad \quad \geq c\int_{\Omega_{t_{1},t}}|u|^{q-1}u\phi_{\delta}(v-u)\sigma_{t_{1}, t}\,dz
\end{align}
and
\begin{align}\label{eq2. of comparison theorem}
\int_{\Omega_{t_{1},t}}\partial_{t}(|v|^{q-1}v)\phi_{\delta}(v-u)\sigma_{t_{1}, t}\,dz&+\int_{\Omega_{t_{1},t}}|\nabla v|^{p-2}\nabla v \cdot \nabla (\phi_{\delta}(v-u)\sigma_{t_{1}, t})\,dz \notag \\
&\quad \quad \leq c\int_{\Omega_{t_{1},t}}|v|^{q-1}v\phi_{\delta}(v-u)\sigma_{t_{1}, t}\,dz.
\end{align}
Notice that
\begin{equation*}
\nabla \phi_{\delta}(v-u)=
\begin{cases}
\frac{\,1\,}{\delta}(\nabla v-\nabla u)\quad &0<v-u<\delta \\
0 \quad &\textrm{otherwise}
\end{cases}
\end{equation*}
and thus,
\begin{equation*}
\nabla \phi_{\delta}(v-u)=\frac{\,1\,}{\delta}(\nabla v-\nabla u)\chi_{\{0<v-u<\delta\}}.
\end{equation*}
Subtract (\ref{eq1. of comparison theorem}) from (\ref{eq2. of comparison theorem}) in Lemma \ref{Barrett-Liu} to obtain
\begin{align}
\label{eq3. of comparison theorem}
\lefteqn{
\int_{\Omega_{t_{1},t}}
\partial_{t}\big(|v|^{q-1}v-|u|^{q-1}u \big) \phi_{\delta}(v-u)\sigma_{t_{1}, t}\,dz
} \notag
\qquad &
\\ \notag
&
\leq
- \int_{\Omega_{t_{1},t}}\big(|\nabla v|^{p-2}\nabla v-|\nabla u|^{p-2}\nabla u \big)\cdot \frac{\,1\,}{\delta}(\nabla v-\nabla u)\chi_{\{0<v-u<\delta\}}\sigma_{t_{1}, t}\,dz
\\
& \quad
+ c\int_{\Omega_{t_{1},t}}\big(|v|^{q-1}v-|u|^{q-1}u\big)\phi_{\delta}(v-u)\sigma_{t_{1}, t}\,dz.
\end{align}
We find that the first term on the right hand side of (\ref{eq3. of comparison theorem}) is bounded above as
\begin{equation}\label{eq4. of comparison theorem}
- \frac{c'}{\delta}\int_{\Omega_{t_{1},t}}\chi_{\{0<v-u<\delta\}}|\nabla v-\nabla u|^{p}\sigma_{t_{1}, t}\,dz \leq 0
\end{equation}
for a positive constant $c'$. Thus (\ref{eq3. of comparison theorem}) and (\ref{eq4. of comparison theorem}) lead to
\begin{multline}\label{eq5. of comparison theorem}
\int_{\Omega_{t_{1},t}}\partial_{t}\big(|v|^{q-1}v-|u|^{q-1}u \big) \phi_{\delta}(v-u)\sigma_{t_{1}, t}\,dz
\\
\leq c \int_{\Omega_{t_{1},t}}\big(|v|^{q-1}v-|u|^{q-1}u\big)\phi_{\delta}(v-u)\sigma_{t_{1}, t}\,dz.
\end{multline}
Since $\partial_{t}(|u|^{q-1}u)$ and $\,\partial_{t}(|v|^{q-1}v) $ belong to  $L^{2}(\Omega_T)$, by the Lebesgue's dominated convergence theorem, we can take the limit as $\delta \searrow 0$ in (\ref{eq5. of comparison theorem}) and then, as $t_1 \searrow 0 $ to obtain
\begin{equation*}
\int_{\Omega}\big(|v|^{q-1}v(t)-|u|^{q-1}u(t) \big)_{+}\,dx \leq c\int_{0}^{t}\int_{\Omega} \big(|v|^{q-1}v(\tau)-|u|^{q-1}u(\tau) \big)_{+}\,dxd\tau,
\end{equation*}
where we used that $\phi_{\delta}(v-u) \to \chi_{\{v>u\}}$ as $\delta \searrow 0$ and that, from (\ref{BL2pre}), $u \geq v$ is equivalent to $|u|^{q-1} u \geq |v|^{q-1} v$ and, by the initial trace condition,
\begin{equation*}
\lim_{t_1 \rightarrow 0}
\int_\Omega \left(|v|^{q-1}v (t_1)-
|u|^{q-1} u (t_1)
\right)_+dx=0.
\end{equation*}
Thus Gronwall's lemma yields that
\begin{equation*}
\int_{\Omega}\big(|v|^{q-1}v(t)-|u|^{q-1}u(t) \big)_{+}\,dx\leq 0
\end{equation*}
and thus, $|v|^{q-1}v(t) \leq |u|^{q-1}u(t)$ in $\Omega, 0 \leq t \leq T$, which is equivalent to that
$v(t) \leq u(t)$ in $\Omega$, $0 \leq t \leq T$. Hence the proof is complete.
\end{proof}

\subsection{Caccioppoli type estimates}

We present the Caccioppoli type estimates, which have a crucial role in De Giorgi's method (see Section \ref{Expansion of positivity section}). From Proposition \ref{Nonnegativity prop} we find that if $u_0 \geq 0$ in $\Omega$, a weak solution $u$ of (\ref{maineq}) is nonnegative in $\Omega_T$. Thus we can consider
(\ref{maineq}) as
\begin{equation}\label{maineq'}
\begin{cases}
\,\partial_tu^q-\Delta_pu=cu^q\quad &\textrm{in}\,\,\Omega_T,\\ \tag{3.1'}
\, 0 \leq u \leq M \quad &\textrm{on}\,\,\partial_{p}\Omega_T
\end{cases}
\end{equation}
In what follows, we always assume that $u_0 \geq 0$ in $\Omega$ and address (\ref{maineq'}) in place of (\ref{maineq}).

Let $K$ be a subset compactly contained in $\Omega$, and $0< t_1 < t_2 \leq T$. Here we use the notation $K_{t_1, t_2} = K \times (t_1, t_2)$. Let $\zeta$ be a smooth function such that $0 \leq \zeta \leq 1$ and $\zeta = 0$ outside $K_{t_1, t_2}$. By use of $A^+(k,u)$ and $A^{-}(k,u)$, the local energy inequality can be derived.
\begin{lem}\label{energy lemma}
Let $k\geq 0$. Then following holds true:
\begin{enumerate}[(i)]
\item Let u be a nonnegative weak supersolution to (\ref{maineq'}). Then there exists a positive constant $C$ depending only on $p,n$ such that
\begin{align}\label{local energy ineq}
\lefteqn{
\esssup_{t_{1}<t<t_{2}}\int_{K \times \{t\}}A^{-}(k,u)\zeta^{p}\,dx+\int_{K_{t_1,t_2}}|\nabla (k-u)_{+}\zeta |^{p}dz} \qquad &
\notag \\
&  \leq C\int_{K \times \{t_{1}\}}A^{-}(k,u)\zeta^{p}\,dx+C\int_{K_{t_1,t_2}}(k-u)_{+}^{p}|\nabla \zeta|^{p}\,dz \notag \\
&\qquad +C\int_{K_{t_1,t_2}}A^{-}(k,u)\zeta^{p-1}|\zeta_{t}|\,dz.
\end{align}
\item Let u be a nonnegative weak subsolution to (\ref{maineq'}). Then there exists a positive constant $C$ depending only on $p,n$ such that
\begin{align}\label{local energy ineq+}
\lefteqn{
\esssup_{t_{1}<t<t_{2}}\int_{K \times \{t\}}A^{+}(k,u)\zeta^{p}\,dx+\int_{K_{t_1,t_2}}|\nabla (u-k)_{+}\zeta |^{p}dz
} \qquad &
\notag \\
&\leq C\int_{K \times \{t_{1}\}}A^{+}(k,u)\zeta^{p}\,dx+C\int_{K_{t_1,t_2}}(u-k)_{+}^{p}|\nabla \zeta|^{p}\,dz \notag \\
&\qquad +C\int_{K_{t_1,t_2}}A^{+}(k,u)\zeta^{p-1}|\zeta_{t}|\,dz+C\int_{K_{t_1,t_2}}cu^q(u-k)_+\zeta^{p}\,dz.
\end{align}
\end{enumerate}
 \end{lem}

\begin{proof} We give the proof only for the case (i), because the case (ii) is treated by a similar argument. We note by (D1) in Definition \ref{def} and the nonnegativity of $u$ in $\Omega_T$ that $\partial_t u^q \in L^2 (\Omega_T)$. Choose a test function $\varphi$ as $-(k-u)_+\zeta^p$ in (D2) in Definition \ref{def} to have
\begin{multline}\label{energy ineq1 lemma}
-\int_{K_{t_1,t}}\frac{\partial u^q}{\partial t}(k-u)_+\zeta^p\,dz -\int_{K_{t_1,t}}|\nabla u|^{p-2}\nabla u \cdot \nabla \big((k-u)_+\zeta^p \big)dz  \\
 =-c\int_{K_{t_1,t}}u^q(k-u)_+\zeta^p\,dz  \leq 0.
\end{multline}
Using the formula (\ref{eq. of A^{-}}), the first term on the left hand side of (\ref{energy ineq1 lemma}) is computed as
\begin{align}\label{energy ineq2 lemma}
-\int_{K_{t_1,t}}\frac{\partial u^q}{\partial t}(k-u)_+\zeta^p\,dz &= \int_{K_{t_1,t}}\frac{\partial}{\partial t}A^{-}(k,u)\zeta^p\,dz \notag \\
&=\int_KA^{-}(k,u)\zeta^p\,dx \bigg|_{t_1}^t-p\int_{K_{t_1,t}}A^{-}(k,u)\zeta^{p-1}|\zeta_t|\,dz.
\end{align}
By use of Young's inequality, the second term on the left hand side of (\ref{energy ineq1 lemma}) is estimated from below by
\begin{align}\label{energy ineq3 lemma}
 \frac{1}{2} \int_{K_{t_1,t}}|\nabla (k-u)_+|^p\zeta^p\,dz-C\int_{K_{t_1,t}}(k-u)_+^{p}|\nabla\zeta|^p\,dz.
\end{align}
We gather (\ref{energy ineq1 lemma}), (\ref{energy ineq2 lemma}) and (\ref{energy ineq3 lemma}) to obtain, for any $t \in (t_1,t_2)$,
\begin{align}
\lefteqn{
\int_{K\times \{t\}}A^-(k,u)\zeta^p\,dx+\int_{K_{t_1,t}}|\nabla(k-u)_+|^p\zeta^p\,dz
} \qquad &
\notag \\
& \leq C\int_{K\times \{t_1\}}A^-(k,u)\zeta^p\,dx+C\int_{K_{t_1,t}}A^-(k,u)\zeta^{p-1}|\zeta_t|\,dz \notag \\
&\qquad +C\int_{K_{t_1,t}}(k-u)_+|\nabla\zeta|^p\,dz.
\end{align}
Thus, we arrive at the conclusion.
\end{proof}

The following so-called Caccioppoli type estimate follows from Lemma \ref{energy lemma}.
\begin{prop}[Caccioppoli type estimate]
Let $k\geq 0$. Let $u$ be a nonnegative weak supersolution of (\ref{maineq}). Then, there exists a positive constant $C$ depending only on $p,n$ such that
\begin{align}\label{local energy ineq}
\lefteqn{
\esssup_{t_{1}<t<t_{2}}\int_{K \times \{t\}}(k-u)_{+}^{q+1}\zeta^{p}\,dx+\int_{K_{t_1,t_2}}|\nabla (k-u)_{+}\zeta |^{p}dz
} \qquad &
\notag \\
&\leq C\int_{K \times \{t_{1}\}}k^{q-1}(k-u)_{+}^{2}\zeta^{p}\,dx+C\int_{K_{t_1,t_2}}(k-u)_{+}^{p}|\nabla \zeta|^{p}\,dz \notag \\
&\qquad +C\int_{K_{t_1,t_2}}k^{q-1}(k-u)_{+}^{2}|\zeta_{t}|\,dz.
\end{align}
\end{prop}

\begin{proof}
We first estimate $\displaystyle A^-(k,u)=q\int_u^k(k-\eta)_+\eta^{q-1}\,d\eta$ defined in (\ref{def of A}).
The lower boundedness is obtained as follows:\\
Case 1 ($u \geq k/2$): \quad Since $\frac{k}{2} \leq u \leq \eta \leq k$ i.e., $\eta \geq k-\eta \geq 0$, it holds that
\begin{equation}\label{Caccioppoli estimate0}
A^-(k,u)\geq q \int_{u}^{k}(k-\eta)^{q}\,d\eta=\frac{q}{q+1}(k-u)^{q+1}.
\end{equation}
\noindent\\
Case 2 ($u \leq k/2$): \quad Since $\frac{k}{2} \leq \eta \leq k$, i.e., $0\leq k-\eta \leq \eta$, it holds that
\begin{align}\label{Caccioppoli estimate1}
A^-(k,u)&=q\int_u^{k/2}(k-\eta)_+\eta^{q-1}\,d\eta+q\int_{k/2}^{k}(k-\eta)_+\eta^{q-1}\,d\eta \notag \\
& \geq q\int_{k/2}^{k}(k-\eta)^{q}\,d\eta  \notag \\
& =\frac{q}{q+1}\left(\frac{k}{2}\right)^{q+1} \notag \\
&\geq \frac{q}{q+1}\frac{1}{2^{q+1}}(k-u)^{q+1},
\end{align}
where, in the last line, we use $k >k-u \geq 0$ since $0 \leq u \leq k/2$.
Also, the upper boundedness follows from
\begin{align}\label{Caccioppoli estimate2}
A^-(k,u)&=q \int_0^{(k-u)_+}(k-\eta)^{q-1}\eta\,d\eta \leq qk^{q-1}\int_0^{(k-u)_+}\eta\,d\eta=qk^{q-1}\frac{(k-u)_+^2}{2}.
\end{align}
Gathering Lemma \ref{energy lemma}, (\ref{Caccioppoli estimate0}), (\ref{Caccioppoli estimate1}) and (\ref{Caccioppoli estimate2}), we arrive at the conclusion.
\end{proof}


\section{Expansion of positivity}\label{Expansion of positivity section}

In this section, we will establish the expansion of positivity of a nonnegative solution to the doubly nonlinear equations of $p$-Sobolev flow type (\ref{maineq'}).

\smallskip

We make local estimates to show the expansion on space-time of positivity of a nonnegative weak (super)solution of (\ref{maineq'}). For any positive numbers $\rho$, $\tau$ and any point $z_0=(x_{0},t_{0}) \in \Omega_{T}$, we denote a local parabolic cylinder of radius $\rho$ and height $\tau$ with vertex at $z_0$ by
\begin{equation*}
Q(\tau,\rho)(z_0):=B_{\rho}(x_{0})\times (t_{0}-\tau,t_{0}).
\end{equation*}
For brevity, we write $Q(\tau,\rho)$ as  $Q(\tau,\rho)(0)$.
%
\subsection{Expansion of interior positivity I}
In this subsection we will study expansion of local positivity of a weak solution of (\ref{maineq'}). Following the argument in \cite{DiBenedetto2} (see also \cite{DiBenedetto1},~\cite{KuusiPisa},~\cite{Urbano}), we proceed our local estimates.
\begin{prop}\label{ex. prop}
Let $u$ be a nonnegative weak supersolution of (\ref{maineq'}). Let $B_\rho (x_0) \subset \Omega$
with center $x_0 \in \Omega$ and radius $\rho > 0$, and $t_0 \in (0, T]$. Suppose that
\begin{equation}\label{ex. propassumption}
\big|B_\rho(x_0) \cap \{u(t_0) \geq L\} \big| \geq \alpha |B_\rho|
\end{equation}
holds for some $L>0$ and $\alpha \in (0,1]$. Then there exists positive numbers $\delta,\,\varepsilon \in (0,1)$ depending only on $p,n$ and $\alpha$ and independent of $L$ such that
\begin{equation}\label{ex. propconc}
\big|B_\rho(x_0) \cap \{u(t) \geq \varepsilon L\}\big| \geq \frac{\alpha}{2} |B_\rho|
\end{equation}
for all $t \in [t_0,\,t_0+\delta L^{q+1-p}\rho^p ]$.
\end{prop}

\begin{proof}
By a parallel translation invariance of the equation (\ref{maineq'}) we may assume $(x_{0},t_{0})=(0,0)$. For $k>0$ and $t>0$, let
\begin{equation*}
A_{k,\rho}(t):=B_\rho \cap \{u(t)<k\}.
\end{equation*}
 Take a cutoff function $\zeta=\zeta(x)$ satisfying
 \begin{equation*}
\xi \equiv 1\quad \textrm{on}\,\,B_{(1-\sigma)\rho},\quad |\nabla \zeta | \leq \frac{1}{\sigma \rho},
\end{equation*}
where $\sigma $ is to be determined later. Applying the Caccioppoli type inequality (\ref{local energy ineq}) over $Q^+(\theta \rho^p,\rho):=B_\rho \times (0,\theta\rho^p)$ to the truncated function $(L-u)_{+}$ and above $\zeta$, we obtain, for any $t \in (0, \theta \rho^p)$,
\begin{align}
&\int_{B_{(1-\sigma)\rho}}(L-u(t))_+^{q+1}\,dx+C\int_{Q^+(\theta \rho^p,\rho)}|\nabla (L-u)_+ \zeta|^p\,dz \notag \\
& \qquad \leq \int_{B_\rho} L^{q-1}(L-u(0))_+^2\,dx+C\int_{Q^+(\theta \rho^p,\rho)}(L-u)_+^p \left(\frac{1}{\sigma \rho} \right)^p\,dz \notag \\
&\qquad \leq L^{q+1}(1-\alpha)|B_\rho|+C\frac{\theta}{\sigma^p}L^p|B_\rho|,
\end{align}
where we use the assumption (\ref{ex. propassumption}) for $u(0)$, and thus, for any $t \in (0, \theta \rho^p)$,
\begin{equation}\label{ex. propineq1}
\int_{B_{(1-\sigma)\rho}}(L-u(t))_+^{q+1}\,dx \leq L^{q+1}\left\{(1-\alpha)+C\frac{\theta L^{p-(q+1)}}{\sigma^p} \right\}|B_\rho|.
\end{equation}
We will estimate the left hand side of (\ref{ex. propineq1}). Firstly, we obtain that
\begin{align}\label{ex. propineq2}
\int_{B_{(1-\sigma)\rho}}(L-u(t))_+^{q+1}\,dx &\geq \int_{B_{(1-\sigma)\rho} \,\cap \{u(t)<\varepsilon L\}}(L-u)_+^{q+1}\,dx \notag \\
& \geq L^{q+1}(1-\varepsilon)^{q+1}\big|A_{\varepsilon L,(1-\sigma)\rho}(t)\big|.
\end{align}
Since $A_{\varepsilon L,\rho}(t) \,\backslash \,A_{\varepsilon L,(1-\sigma)\rho}(t) \subset  B_\rho \,\backslash \,B_{(1-\sigma)\rho}$ and
\begin{align}
|B_\rho \,\backslash \,B_{(1-\sigma)\rho}|&=|B_1(0)|\rho^n\{1-(1-\sigma)^n\} \notag \\
&\leq |B_1(0)|\rho^n \{1-(1-n\sigma)\} \notag \\
& \leq n\sigma |B_\rho|,\notag
\end{align}
we have
\begin{align}\label{ex. propineq3}
|A_{\varepsilon L,\rho}(t)| &\leq |A_{\varepsilon L,(1-\sigma)\rho}(t)| +|A_{\varepsilon L,\rho}(t) \,\backslash \,A_{\varepsilon L,(1-\sigma)\rho}(t)| \notag \\
& \leq |A_{\varepsilon L,(1-\sigma)\rho}(t)|+n\sigma |B_\rho|.
\end{align}
From (\ref{ex. propineq2}) and (\ref{ex. propineq3}), we have
\begin{equation}\label{ex. propineq4}
\int_{B_{(1-\sigma)\rho}}(L-u(t))_+^{q+1}\,dx \geq L^{q+1}(1-\varepsilon)^{q+1} \left(|A_{\varepsilon L,\rho}(t)|-n\sigma |B_\rho| \right).
\end{equation}
and thus, (\ref{ex. propineq1}) and (\ref{ex. propineq4}) yield that, for any $t \in (0, \theta \rho^p)$,
\begin{equation}\label{ex. propineq5}
|A_{\varepsilon L,\rho}(t)| \leq \frac{1}{(1-\varepsilon)^{q+1}}\left\{(1-\alpha)+C\frac{\theta L^{p-(q+1)}}{\sigma^p}+n\sigma \right\}|B_\rho|.
\end{equation}

Here we choose the parameters as
\begin{equation}\label{parameter}
\theta=\delta L^{q+1-p},\quad \delta=\frac{\alpha^{p+1}}{2^{3p+3}Cn^p}, \quad \sigma=\frac{\alpha}{8n}
\end{equation}
and $\varepsilon$ with
\begin{equation}\label{parameter2}
\varepsilon \leq 1-\Bigg\{\frac{1-\frac{3}{4}\alpha}{1-\frac{1}{2}\alpha} \Bigg\}^{\frac{1}{q+1}}.
\end{equation}
Then we find from (\ref{ex. propineq5})-(\ref{parameter2}) that, for any $t \in [0, \theta \rho^p]$,
\begin{align}
|A_{\varepsilon L,\rho}(t)| \leq \frac{1}{\,\frac{1-\frac{3}{4}\alpha}{1-\frac{1}{2}\alpha}\,}\cdot \left(1-\frac{3}{4}\alpha \right)|B_\rho|=\left(1-\frac{1}{2}\alpha \right)|B_\rho|,
\end{align}
that is, (\ref{ex. propconc}) is actually verified under (\ref{parameter}), and thus the proof is complete.
\end{proof}

\begin{lem}\label{ex. crucial lemma}
Let $u$ be a nonnegative weak supersolution of (\ref{maineq'}).
Let $Q_{4\rho}(z_{0}):=B_{4\rho}(x_{0}) \times (t_{0}, t_{0}+\delta L^{q+1-p}\rho^p) \subset \Omega_{T}$, where $\delta$ is selected in Proposition \ref{ex. prop}. Then for any $\nu \in (0,1)$ there exists a positive number $\varepsilon_\nu$ depending only on $p,n,\alpha,\delta,\nu$ such that
\begin{equation*}
\big|Q_{4\rho}(z_{0}) \cap \{u<\varepsilon_\nu L\} \big| <\nu \big|Q_{4\rho}\big|.
\end{equation*}
\end{lem}

\begin{proof}
We may assume $z_0=0$ as before. By Proposition \ref{ex. prop}, there exist positive numbers $\delta,\,\varepsilon \in (0,1)$ such that
\begin{equation}\label{lemma 4.2 ineq1}
\big|B_{4\rho} \cap \{u(t) \geq \varepsilon L\} \big| \geq \frac{\alpha}{2} \cdot 4^{-n}|B_{4\rho}|
\end{equation}
for all $t \in [0,\,\delta L^{q+1-p}\rho^p]$.

Set  $\theta=\delta L^{q+1-p}$ and let $\zeta=\zeta(x)$ be a piecewise smooth cutoff function satisfying
\begin{align}\label{cutoff lemma 4.2}
&0 \leq \zeta \leq 1,\quad \zeta \equiv 0\quad \textrm{oustside}\quad B_{8\rho}\notag \\
&\zeta \equiv 1\quad \mathrm{on}\quad Q_{4\rho},\quad |\nabla \zeta| \leq (4\rho)^{-1}.
\end{align}
From (\ref{cutoff lemma 4.2}) and the Caccioppoli type inequality (\ref{local energy ineq}), applied for the truncated solution $(k_j - u)_+$ over $Q_{4 \rho}$ with the level $k_j=\frac{1}{2^j}\varepsilon L\,\,(j=0,1,\ldots)$, we obtain
\begin{align}\label{ex. positivity ineq1}
\int_{Q_{4\rho}}|\nabla (k_j-u)_+|^p\zeta^p\,dz &\leq \int_{B_{8 \rho} \times \{t = 0\}}k_j^{q-1} (k_j-u)_+^2 \zeta^p\,dx +C\int_{Q_{8\rho}}(k_j-u)_+^p|\nabla \zeta|^p \,dz\notag \\
 & \leq C \left(k_j^{q+1}|B_{8 \rho}|+k_j^p |Q_{8 \rho}| (4\rho)^{-p}\right)  \notag \\
&\leq C k_j^p L^{q + 1 - p} |B_{8 \rho}|\left(1+2^{-p} \delta\right)\notag \\
&\leq C\frac{k_j^p}{\delta \rho^{p}}\big|Q_{8\rho}\big|
=C\frac{k_j^p}{\delta \rho^{p}}\big|Q_{4\rho} \big|.
\end{align}
Here we note that the constant $C$ depends only on $n$, $p$ and,  in particular, is independent of $\rho, L$.
On the other hand, applying De Giorgi's inequality (\ref{De Giorgi's ineq.}) in Proposition \ref{De Giorgi's inequality} to $k=k_{j+1}$ and $l=k_j$, we have, for all $t$, $0\leq t \leq \delta L^{q+1-p}\rho^{p}=\theta \rho^{p}$,
\begin{multline}\label{ex. positivity ineq2}
(k_{j}-k_{j+1})\big|B_{4\rho} \cap \{u(t)>k_j\}\big|
\\
\leq C\frac{\rho^{n+1}}{\big|B_{4\rho} \cap \{u(t) <k_{j+1}\} \big|}\int_{B_{4\rho} \cap \{k_{j+1}<u(t) <k_j\}}|\nabla u(t)|\,dx.
\end{multline}
Let $A_{j}(t):=B_{4\rho} \cap \{u(t)<k_{j}\}$ and then, by (\ref{lemma 4.2 ineq1}), it holds that
\begin{equation}\label{ineq A_{j}}
\big|B_{4\rho} \setminus A_{j}(t)\big| \geq \frac{\alpha}{2}4^{-n}|B_{4\rho}|.
\end{equation}
Combine (\ref{ineq A_{j}}) with (\ref{ex. positivity ineq2}) to have that
\begin{align}\label{ex. positivity ineq3}
\frac{k_{j}}{2}\big|A_{j+1}(t)\big| &\leq \frac{\rho^{n+1}}{|B_{4\rho} \setminus A_{j}(t)|}\int_{B_{4\rho} \cap \{k_{j+1}<u(t) <k_j\}}|\nabla u(t)|\,dx \notag \\
&\leq \frac{C}{\alpha}\rho \int_{B_{4\rho} \cap \{k_{j+1}<u(t) <k_j\}}|\nabla u(t)|\,dx.
\end{align}
Integrating above inequality (\ref{ex. positivity ineq3}) in $t \in (0,\theta\rho^{p})$ yields
\begin{equation}\label{ex. positivity ineq4}
\frac{k_{j}}{2}|A_{j+1}| \leq \frac{C}{\alpha}\rho \int_{Q_{4\rho} \cap \{k_{j+1}<u<k_j\}}|\nabla u|\,dz,
\end{equation}
where we put $\displaystyle |A_{j}|:=\int_{0}^{\theta \rho^{p}}|A_{j}(t)|\,dt=\big|Q_{4\rho}\cap \{u(t) <k_{j}\}\big|$. By use of H\"older's inequality,  (\ref{ex. positivity ineq1}) and (\ref{ex. positivity ineq4}), we have
\begin{align}\label{ex. positivity ineq5}
\frac{k_{j}}{2}|A_{j+1}| &\leq \frac{C}{\alpha}\rho \left[\int_{Q_{4\rho}}|\nabla (k_j-u)_+|^p\,dz\right]^\frac{1}{p}\big|A_j \,\backslash\,A_{j+1} \big|^\frac{p-1}{p} \notag \\
& \leq \frac{C}{\alpha}\rho \cdot \frac{k_{j}}{\delta^{\frac{1}{p}}\rho} |Q_{4\rho}|^\frac{1}{p} \left(|A_j|-|A_{j+1}| \right)^\frac{p-1}{p} \notag \\
&=\frac{C}{\alpha \delta^\frac{1}{p}}k_j \big|Q_{4\rho}\big|^\frac{1}{p}\left(|A_j|-|A_{j+1}| \right)^\frac{p-1}{p}
\end{align}
and thus,
\begin{equation}\label{ex. positivity ineq6}
|A_{j+1}|^\frac{p}{p-1} \leq \left(\frac{C}{\alpha \delta^\frac{1}{p}}\right)^\frac{p}{p-1}\big|Q_{4\rho}\big|^\frac{1}{p-1}\left(|A_j|-|A_{j+1}| \right).
\end{equation}
Let $J \in \mathbb{N}$ be determined later. Summing (\ref{ex. positivity ineq6}) over $j=0,1,\ldots, J-1$, we obtain
\begin{equation}\label{ex. positivity ineq7}
J|A_J|^\frac{p}{p-1} \leq \left(\frac{C}{\alpha \delta^\frac{1}{p}}\right)^\frac{p}{p-1}\big|Q_{4\rho}\big|^\frac{p}{p-1}.
\end{equation}
Indeed, by use of $|A_{0}| \geq |A_{j}| \geq |A_{J}|$ for $j \in \{0,1,\ldots, J\}$, we find that
\begin{equation*}
\sum \limits_{j=0}^{J-1}|A_{j+1}|^\frac{p}{p-1} \geq J|A_{J}|^\frac{p}{p-1}\quad \textrm{and} \quad \sum \limits_{j=0}^{J-1}\left(|A_j|-|A_{j+1}| \right) \leq |A_0|\leq \big|Q_{4\rho}\big|.
\end{equation*}%
Therefore, from (\ref{ex. positivity ineq7}), it follows that
\begin{equation}\label{ex. positivity ineq8}
|A_J| \leq \frac{1}{J^\frac{p-1}{p}}\left(\frac{C}{\alpha \delta^\frac{1}{p}}\right)\big|Q_{4\rho}\big|.
\end{equation}
Thus, for any $\nu \in (0,1)$, we choose sufficiently large $J \in \mathbb{N}$ satisfying
\begin{equation}\label{def of J}
\frac{1}{J^\frac{p-1}{p}}\left(\frac{C}{\alpha \delta^\frac{1}{p}}\right) \leq \nu \iff J \geq \left(\frac{C}{\nu\alpha \delta^\frac{1}{p}}\right)^\frac{p}{p-1}.
\end{equation}
Here we note that $J$ depends only on $p, n, \alpha, \delta$ and $\nu$. We finally take $\displaystyle \varepsilon_\nu=\frac{\varepsilon}{2^J}$ and then (\ref{ex. positivity ineq8}) yields that
\begin{equation*}
\frac{\big|Q_{4\rho} \cap \{u<\varepsilon_\nu L\}\big| }{\big|Q_{4\rho} \big|}<\nu,
\end{equation*}
which is the very assertion.
\end{proof}

\begin{rmk}\normalfont \label{choice of parameter}
Look at the choice of $\delta$ in (\ref{parameter}) and $\varepsilon$ in (\ref{parameter2}) in the proof of Proposition \ref{ex. prop}, from which we can choose $\varepsilon$ such that
\begin{equation}\label{choice of epsilon}
\varepsilon=\left(\frac{\delta}{2^I}\right)^{\frac{1}{q+1-p}}
\end{equation}
 for some large positive integer $I$. In the proof of Lemma \ref{ex. crucial lemma} and the choice of $k_j$, we also choose $k_j$ as follows:
\begin{equation}\label{choice of k_j}
 k_j=\displaystyle \frac{\varepsilon L} {2^{\frac{j}{q+1-p}}}\quad \textrm{for} \quad j=0,1,\ldots,J.
\end{equation}
Under such choice as above we note that $\displaystyle k_J=\left(\tfrac{\delta}{2^{I+J}}\right)^{\frac{1}{q+1-p}}L$ and obtain that $$\frac{\delta L^{q+1-p}\rho^p}{(k_J)^{q+1-p} \rho^p}=2^{J+I},$$ which is a positive integer.\\
 \quad Following a similar argument to \cite[p.76]{DiBenedetto2}, we next divide $Q_{4\rho}(z_{0})$ into finitely many subcylinders. For any $\nu \in (0,1)$, let $J$ be determined in (\ref{def of J}). We divide $Q_{4\rho}(z_{0})$ along time direction into parabolic cylinders of number $s_{0}:=2^{I+J}$ with each time-length $k_{J}^{q+1-p}\rho^{p}$, and set
\begin{equation*}
Q^{(\ell)}:=B_{4\rho}(x_{0}) \times \left(t_{0}+\ell k_{J}^{q+1-p}\rho^{p},\,\,t_{0}+(\ell+1) k_{J}^{q+1-p}\rho^{p} \right)
\end{equation*}
for $\ell=0,1,\ldots ,s_{0}-1$.
\begin{center}
\begin{tikzpicture}[scale=1.4]
\draw[thin] (-1.0,0) -- (1.0,0);
\draw[thin] (-1.0,0.5) -- (1.0,0.5);
\draw (-1,-0.5) rectangle (1.0,1.0);
\draw[densely dotted] (-1.5,1.0) -- (-1.0,1.0);
\draw[densely dotted] (-1.5,0.5) -- (-1.0,0.5);
\draw[densely dotted] (-1.5,0) -- (-1.0,0);
\draw[densely dotted] (-1.5,-0.5) -- (-1.0,-0.5);
\draw[->] (-1.5,-1.0) -- (1.5, -1.0)node[right]{\footnotesize $x \in \mathbb{R}^{n}$};
\draw[->] (-1.50,-1.4)-- (-1.5,1.5)node[right]{\footnotesize $t$};
\draw[thick, <->] (1.5,0)-- (1.5,0.5);
\draw (2.8, 0.2)node{\footnotesize height: $(k_{J})^{q+1-p}\rho^{p}$};
\filldraw [pattern=north east lines, pattern color=blue,] (-0.99,0.01) rectangle (0.99,0.49);
\draw  (0.4,0.2)node[left]{\footnotesize $Q^{(\ell)}$};
\draw  (-1.5,1.0)node[left]{\footnotesize $t_{0}+\delta L^{q+1-p}\rho^{p}$};
\draw  (-1.5,0.5)node[left]{\footnotesize $t_{0}+(\ell+1) k_{J}^{q+1-p}\rho^{p}$};
\draw  (-1.5,0)node[left]{\footnotesize $t_{0}+\ell k_{J}^{q+1-p}\rho^{p}$};
\draw  (-1.5,-0.5)node[left]{\footnotesize $t_{0}$};
\draw  (0,-0.5)node[below]{\footnotesize $B_{4\rho}(x_{0})$};
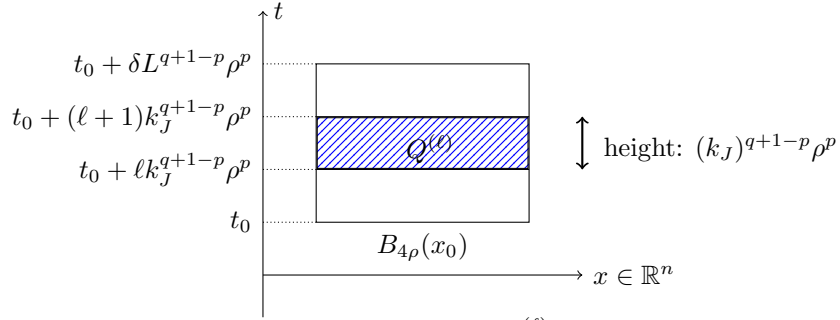
\draw (0,-1.3)node[below] {\footnotesize  Figure: Image of $Q^{(\ell)}$};
\end{tikzpicture}
\end{center}
Then there exists a $Q^{(\ell)}$ which holds that
\begin{equation}\label{Q^l}
|Q^{(\ell)} \cap \{u<k_{J}\}| <\nu |Q^{(\ell)}|.
\end{equation}
\end{rmk}
The following theorem enables us to have the positivity of a solution of (\ref{maineq'}) in a small interval.

\begin{thm}[Expansion of local positivity]\label{Expansion of positivity}
Let $u$ be a nonnegative weak supersolution of (\ref{maineq'}). Let $B_\rho (x_0) \subset \Omega$ with center $x_0 \in \Omega$ and radius $\rho > 0$, and $t_0 \in (0, T]$. Suppose that (\ref{ex. propassumption}). Under (\ref{Q^l}) there exists a positive number $\eta<1$ such that
\begin{equation}\label{conc. Q^l}
u \geq \eta L\quad a.e.\quad B_{2\rho}(x_{0})\times \left(t_{0}+\left(\ell+\frac{1}{2} \right) k_J^{q+1-p} \rho^p,\,\,t_{0}+(\ell+1) k_{J}^{q+1-p}\rho^{p}\right).
\end{equation}
\end{thm}
\begin{proof}
Hereafter we fix the parameters $\rho,\ell$ and $k_J$.  By translation we may assume to shift $(x_0, t_0 + (\ell+1) k_J^{q+1-p} \rho^p)$
to the origin and thus, $Q^{(l)}$ is transformed to $B_{4 \rho} (x_0) \times
(- k_J^{q +1-p}\rho^p, 0)$. For $m=0,1,2,\ldots $, let
\begin{equation*}
\tau_m =\frac{\rho^p}{2} \left(1 + \frac{1}{2^m}\right),\quad \rho_m =\frac{\rho}{2} \left(1 + \frac{1}{2^m}\right);\quad B_{m}:=B_{4\rho_{m}},\quad Q_{m}:= B_m \times (-\theta \tau_m, 0),
\end{equation*}
where $\theta:=k_{J}^{q+1-p}$,
and also set
\begin{equation*}
\kappa_{m}:=\left(\frac{1}{2}+\frac{1}{2^{m+1}} \right)k_{J}.
\end{equation*}
It plainly holds true that
\begin{equation*}
\begin{cases}
\rho^{p}=\tau_{0}\geq \tau_{m} \searrow \tau_{\infty}=\rho^{p}/2,\quad \rho =\rho_{0}\geq \rho_m \searrow \rho_{\infty}=\rho/2 \\[2mm]
Q_0 = Q^{(\ell)}
\supset Q_m \searrow Q_{\infty}=B_{2\rho} \times (-\tfrac12 \theta \rho^{p},0) \\[2mm]
k_{J}=\kappa_{0} \geq \kappa_m \searrow \kappa_{\infty}=k_{J}/2
\end{cases}
\end{equation*}
The cutoff function $\zeta$ is taken of the form $\zeta(x,t)=\zeta_{1}(x)\zeta_{2}(t)$, where $\zeta_i \,(i =1,2)$ are Lipschitz functions satisfying
\begin{equation*}
\zeta_{1}:=\begin{cases}
1\quad &\textrm{in}\quad  B_{m+1}\\
0\quad & \textrm{in}\quad  \mathbb{R}^{n}\,\backslash\,B_{m}
\end{cases}
,\qquad |\nabla \zeta_{1}| \leq \frac{1}{\rho_{m}-\rho_{m+1}}=\frac{2^{m+2}}{\rho}
\end{equation*}
and
\begin{equation*}
\zeta_{2}:=\begin{cases}
0\quad &\textrm{for}\quad t \leq-\theta\tau_{m}\\
1\quad &\textrm{for}\quad t> -\theta \tau_{m+1}
\end{cases}
,\qquad 0 \leq \zeta_{2,t}\leq \frac{1}{\theta(\tau_{m}-\tau_{m+1})} \leq \frac{2^{p(m+2)}}{\theta \rho^{p}}.
\end{equation*}
Thus, applying the local energy inequality (\ref{local energy ineq}) over $B_{m}$ and $Q_{m}$ to the truncated solution $(\kappa_{m}-u)_{+}$ and above $\zeta$, we obtain
\begin{align}\label{eneiq1}
\esssup_{-\theta \tau_{m}<t<0}\int_{B_{m}}&(\kappa_{m}-u(t))_{+}^{q+1}\zeta^{p}\,dx+\int_{Q_{m}}|\nabla (\kappa_{m}-u)_{+}|^{p}\zeta^{p}\,dz \notag \\
&\leq C\int_{Q_{m}}(\kappa_{m}-u)_{+}^{p}|\nabla \zeta|^{p}\,dz+C\int_{Q_{m}}\kappa_{m}^{q-1}(\kappa_{m}-u)_+^{2}|\zeta_{t}|\,dz \notag \\
&\leq C\left(\frac{2^{m+2}}{\rho}\right)^{p}\kappa_{m}^{p}\int_{Q_{m}} \left(1+\frac{\kappa_{m}^{q+1-p}}{\theta} \right)\chi_{\{(\kappa_{m}-u)_{+}>0\}}\,dz \notag \\
&\leq C\left(\frac{2^{m+2}}{\rho}\right)^{p}\kappa_{m}^{p}\int_{Q_{m}}\chi_{\{(\kappa_{m}-u)_{+}>0\}}\,dz,
\end{align}
where we used that $\displaystyle \frac{\kappa_{m}^{q+1-p}}{\theta}=\left(\frac{\kappa_{m}}{k_{J}} \right)^{q+1-p}\leq 1$.
Combining Proposition \ref{parabolic embedding1} with (\ref{eneiq1}), we have
\begin{align}\label{eq1 of positivity of Ql}
&\quad \int_{Q_{m}}|(\kappa_{m}-u)_{+}\zeta|^{q+1}\,dz=\int_{Q_{m}}|(\kappa_{m}-u)_{+}\zeta|^{p\frac{n+q+1}{n}}\,dz \notag \\
&\leq C\bigg(\int_{Q_{m}}\big|\nabla[(\kappa_{m}-u)_{+}\zeta]\big|^{p}\,dz \bigg)\bigg(\esssup_{-\theta \tau_{m}<t<0}\int_{B_{m}}|(\kappa_{m}-u(t))_{+}\zeta|^{q+1}\,dz\bigg)^\frac{p}{n} \notag \\
&\leq  C\left(\frac{2^{m+2}}{\rho}\right)^{p(1+\frac{p}{n})}k_{J}^{p(1+\frac{p}{n})}\left(\int_{Q_{m}}\chi_{\{(\kappa_{m}-u)_{+}>0\}}\,dz\right)^{1+\frac{p}{n}},
\end{align}
where we note that $q+1=\frac{p(n+q+1)}{n}$ in the second line.

The left hand side of (\ref{eq1 of positivity of Ql}) is estimated from below as
\begin{align}\label{eq2 of positivity of Ql}
\int_{Q_{m}}[(\kappa_{m}-u)_{+}\zeta]^{q+1}\,dz &\geq \int_{Q_{m}}[(\kappa_{m}-u)_{+}\zeta]^{q+1}\chi_{\{(\kappa_{m+1}-u)_{+}>0\}}\,dz \notag \\
& \geq |\kappa_{m}-\kappa_{m+1}|^{q+1}\int_{Q_{m+1}}\chi_{\{(\kappa_{m+1}-u)_{+}>0\}}\,dz\notag \\
&=\bigg(\frac{k_{J}}{2^{m+2}} \bigg)^{q+1}\int_{Q_{m+1}}\chi_{\{(\kappa_{m+1}-u)_{+}>0\}}\,dz.
\end{align}
Hence, by (\ref{eq1 of positivity of Ql}) and (\ref{eq2 of positivity of Ql}), we have
\begin{align}\label{eq3 of positivity of Ql}
\int_{Q_{m+1}}\chi_{\{(\kappa_{m+1}-u)_{+}>0\}}\,dz \leq C\frac{[2^{p(1+\frac{p}{n})+q+1}]^{m}}{\rho^{p(1+\frac{p}{n})}}k_{J}^{p(1+\frac{p}{n})-(q+1)}\left(\int_{Q_{m}}\chi_{\{(\kappa_{m}-u)_{+}>0\}}\,dz\right)^{1+\frac{p}{n}},
\end{align}
where we compute $C\big(\frac{k_{J}}{2^{m+2}} \big)^{-(q+1)}\left(\frac{2^{m+1}}{\rho}\right)^{p(1+\frac{p}{n})}k_{J}^{p(1+\frac{p}{n})}=C\frac{[2^{p(1+\frac{p}{n})+q+1}]^{m}}{\rho^{p(1+\frac{p}{n})}}k_{J}^{p(1+\frac{p}{n})-(q+1)}$.
Dividing the both side of (\ref{eq3 of positivity of Ql}) by $|Q_{m+1}|>0$, we have
\begin{align}\label{eq4 of positivity of Ql}
\dashint_{Q_{m+1}}\chi_{\{(\kappa_{m+1}-u)_{+}>0\}}\,dz 
\leq C[2^{p(1+\frac{p}{n})+q+1}]^{m}\left(\dashint_{Q_{m}}\chi_{\{(\kappa_{m}-u)_{+}>0\}}\,dz\right)^{1+\frac{p}{n}},
\end{align}
where
\begin{equation*}
\frac{|Q_{m}|^{1+\frac{p}{n}}}{|Q_{m+1}|}
\leq C\rho^{p(1+\frac{p}{n})}(k_{J}^{q+1-p})^{\frac{p}{n}}
\end{equation*}
and $p(1+\frac{p}{n})-(q+1)+(q+1-p)\frac{p}{n}=0$ are used.
Letting $\displaystyle Y_{m}:=\dashint_{Q_{m}}\chi_{\{(\kappa_{m}-u)_{+}>0\}}\,dz$, the above inequality (\ref{eq4 of positivity of Ql}) is rewritten as
\begin{equation*}
Y_{m}\leq Cb^{m}Y_{m}^{1+\frac{p}{n}},\quad m=0,1,\ldots,
\end{equation*}
where $b:=2^{p(1+\frac{p}{n})+q+1}$. From Lemma \ref{Fast geometric convergence}, we find that if the initial value $Y_{0}$ satisfies
\begin{equation}\label{condition of Y_{0}}
Y_{0}\leq C^{-(\frac{n}{p})}b^{-(\frac{n}{p})^{2}}=:\nu_{0},
\end{equation}
then
\begin{equation}\label{conclusion of Y_{m}}
Y_{m} \to 0\quad \textrm{as}\quad m \to \infty.
\end{equation}
Eq.(\ref{Q^l}) is equivalent to (\ref{condition of Y_{0}}) by taking $\nu=\nu_0$, and then (\ref{conclusion of Y_{m}}) leads to the conclusion (\ref{conc. Q^l}) by putting $\eta=\frac{k_{J}}{2L}<1$.
\end{proof}

\begin{rmk}\normalfont
Theorem \ref{Expansion of positivity} asserts that the positivity propagates after the lapse of some time. If a solution is positive at some time $t_0$, its positivity expands in space-time  without ''waiting time'', which is in the next corollary.
\end{rmk}

\begin{cor}\label{key cor}Let $u$ be a nonnegative weak supersolution of (\ref{maineq'}). Assume  that $u(t_0)>0$ almost everywhere in $B_{4 \rho} (x_0) \subset \Omega$. Then there exist positive numbers $\eta_0$ and  $\tau_0$ such that \begin{equation*}
u \geq \eta_0 \quad \textrm{a.e.}\quad  \textrm{in}\quad  B_{2\rho}(x_0) \times (t_0, t_0+\tau_0).
\end{equation*}
\end{cor}

\begin{proof} Suppose that $(x_0, t_0)$ be the origin, as before. Let $L:=\inf_{B_{4 \rho}}u(0)>0$. Since $|B_{\rho} \cap \{u(0) \geq L\} |=|B_{\rho}|$, by Proposition \ref{ex. prop}, there exist positive numbers $\delta, \varepsilon$ depending only on $n,p$ and independent of $L$ such that
\begin{equation}\label{ineq of B_4}
 \big| \{u(t) \geq \varepsilon L\} \cap B_{\rho} \big| \geq \frac{1}{2}|B_{\rho}|
\end{equation}
for all $t \in [0,\delta L^{q+1-p}\rho^p]$.
Let $Q^\theta_{4\rho}(z_0)=B_{4\rho}(x_0) \times (0,\theta \rho^p) \in \Omega_T$, where $0<\theta<\delta L^{q+1-p}\rho^{p}$ is a parameter determined later.  By Lemma \ref{ex. crucial lemma} with some minor change, for any $\nu \in (0,1)$ there exists a positive number $\varepsilon_\nu$ depending only on $p,n,\delta$ and $\nu$ such that
\begin{equation}\label{condition of Q^{theta}}
\big|Q^{\theta}_{4\rho}(z_{0}) \cap \{u<\varepsilon_\nu L\} \big| <\nu \big|Q^{\theta}_{4\rho}\big|.
\end{equation}
Here we notice that in the proof of Proposition \ref{ex. prop} and Lemma \ref{ex. crucial lemma}, we do not need to use the cutoff on time. In the following we modify the proof of Theorem \ref{Expansion of positivity} to that without any cutoff on time.

By translation, we may assume $x_0=0$ as before.  For $m=0,1,\ldots,$ we put %
\begin{equation*}
\rho_m:=\left(2+\frac{1}{2^{m-1}}\right)\rho,\quad Q_m:=B_{\rho_m} \times (0,\theta \rho^p),\quad \theta=\delta k_J^{q+1-p}<\delta L^{q+1-p}
\end{equation*}
where $k_{J}=\frac{\varepsilon L}{2^J}$  and $J$ is to be determined in (\ref{def of J}), and also set
\begin{equation*}
\kappa_{m}:=\left(\frac{1}{2}+\frac{1}{2^{m+1}} \right)k_{J}.
\end{equation*}
Clearly it holds true that
\begin{equation*}
\begin{cases}
4\rho =\rho_{0}\geq \rho_m \searrow \rho_{\infty}=2\rho , \quad
Q_0\supset Q_m \searrow Q_{\infty}=B_{2\rho} \times (0,\theta \rho^{p})\quad  ; \\[2mm]
 k_{J}=\kappa_{0} \geq \kappa_m \searrow \kappa_{\infty}=k_{J}/2.
\end{cases}
\end{equation*}
A cutoff function $\zeta$ is chosen of the form $\zeta(x,t)=\zeta_1(x)$, where %
\begin{equation*}
\zeta_1=\begin{cases}
 1 \quad & x \in B_{\rho_{m+1}}\\
 0 \quad & x \in \mathbb{R}^n \,\backslash\, B_{\rho_m}
 \end{cases}
 ,\quad |\nabla \zeta_1(x)| \leq \frac{1}{\rho_m-\rho_{m+1}}=\frac{2^{m+1}}{\rho^m}.
 \end{equation*}
From the Caccioppoli type inequality (\ref{local energy ineq}), applied for the truncated solution $Q_m$ to the truncated $(\kappa_m-u)_+$ again, we obtain that
\begin{align}\label{ex. positivity key ineq}
&\esssup_{0<t<\theta\rho^{p}}\int_{B_{m}}(\kappa_{m}-u(t))_{+}^{q+1}\zeta^{p}\,dx+\int_{Q_m}|\nabla (\kappa_m-u)_+|^p\zeta^p\,dz \notag\\
&\qquad \leq \int_{B_m}\kappa_m^{q - 1} (\kappa_m - u (0))_+^2 \zeta^p\,dx+ C\int_{Q_m} (\kappa_m-u)_+^p|\nabla \zeta|^p \,dz\notag \\
 & \qquad \leq C \left(\frac{\rho}{2^{m+1}} \right)^{p}\kappa_m^p \int_{Q_m} \chi_{\{(\kappa_m-u)_+>0\}}\,dz,
\end{align}
where, in the second line, we used $(\kappa_{m}-u(0))_{+}=0$ in $B_{m}$. By the very same argument as in the proof of Theorem \ref{Expansion of positivity}, we have that
\begin{align}\label{eq4 of positivity of Q}
\dashint_{Q_{m+1}}\chi_{\{(\kappa_{m+1}-u)_{+}>0\}}\,dz \leq C[2^{p(1+\frac{p}{n})+q+1}]^{m}\left(\dashint_{Q_{m}}\chi_{\{(\kappa_{m}-u)_{+}>0\}}\,dz\right)^{1+\frac{p}{n}}.
\end{align}
Letting $\displaystyle Y_m:=\dashint_{Q_m} \chi_{\{(\kappa_m-u)_+>0\}}\,dz$, the above inequality (\ref{eq4 of positivity of Q}) is rewritten as:
\begin{equation*}
 Y_{m+1} \leq Cb^m Y_m^{1+\frac{p}{n}},\quad m=0,1,\ldots,
 \end{equation*}
 where $b:=2^{p(1+\frac{p}{n})+q+1}>1$.
 By Lemma \ref{Fast geometric convergence} on fast geometric convergence if
 \begin{equation}\label{condition2 on Y_{0}}
 Y_0 \leq (Cd)^{-\frac{n}{p}}b^{-(\frac{n}{p})^2}=:\nu_0
 \end{equation}
 then
 \begin{equation}\label{limit2 Y_{n}}
 Y_m \to 0\quad \textrm{as}\quad  m \to \infty.
 \end{equation}
 Eq. (\ref{condition2 on Y_{0}}) is equivalent to (\ref{condition of Q^{theta}}) by taking $\nu=\nu_0$, and then (\ref{limit2 Y_{n}}) yields that
\begin{equation*}
 u \geq \eta L\quad \textrm{a.e.}\,\,\textrm{in}\,\, B_{2\rho}\times (0,\delta k_{J}^{q+1-p}\rho^{p}),
\end{equation*}
 where $\eta=\frac{1}{2}\frac{\varepsilon }{2^{J}}$. Thus, letting $\eta_{0}:=\eta L$ and $\tau_{0}:=\delta k_{J}^{q+1-p}\rho^{p}$, we reach the conclusion.
%
\end{proof}

\subsection{Expansion of interior positivity II}

We continue to study the expansion of positivity of a nonnegative solution. Let $\Omega^\prime$ be a subdomain contained compactly in $\Omega$. Using Theorem \ref{Expansion of positivity} and a method of chain of finitely many balls as used in Harnack's inequality for harmonic functions, which is so-called \textit{Harnack chain} (see \cite[Theorem 11, pp.32--33]{Evans} and~\cite{AKN,KMN} in the $p$-parabolic setting), we have the following theorem. Here we use the special choice of parameters, as explained before Theorem \ref{Expansion of positivity}.

\begin{thm}\label{ex. thm prime}
Let u be a nonnegative weak supersolution of (\ref{maineq'}). Let $\Omega^\prime$ be a subdomain contained compactly in $\Omega$. Let $t_0 \in (0, T]$. Suppose that
\begin{equation}\label{ex. propassumption prime}
\big|\Omega^\prime \cap \{u(t_0) \geq L\} \big| \geq \alpha |\Omega^\prime|
\end{equation}
holds for some $L>0$ and $\alpha \in (0,1]$. Then there exist positive integer $N=N(\Omega^\prime)$ and positive real number families $\{\delta_m\}_{m=0}^{N},\,\{\eta_{m}\}_{m=1}^{N+1}\subset (0,1),\,\{J_m\}_{m=0}^{N},\,\{I_m\}_{m=0}^{N}\subset \mathbb{N}$ depending on $p, n, \alpha $ and independent of $L$, a time $t_N>t_0$ such that
\begin{equation*}
u \geq \eta_{N+1} L
\end{equation*}
almost everywhere in
\begin{equation*} 
\Omega^{\prime}\times \left(t_N+\left(k+\frac{1}{2}\right)\frac{\delta_N(\eta_NL)^{q+1-p}}{2^{J_N+I_N}}\rho^p,\,\,t_N+(k+1)\frac{\delta_N(\eta_NL)^{q+1-p}}{2^{J_N+I_N}}\rho^p \right)
\end{equation*}
for some $k \in \left\{0,1,\ldots, 2^{J_N+I_N}-1\right\}$, where $t_{N}$ is written as
\begin{equation*}
t_{N}=t_{0}+\sum \limits_{m=1}^{N}\left(\ell+\frac{3}{4}\right)\frac{\delta_{m-1}(\eta_{m-1}L)^{q+1-p}}{2^{J_{m-1}+I_{m-1}}}\rho^p
\end{equation*}
for some $\ell \in \{0,1,\ldots,2^{J_{m-1}+I_{m-1}-1}\}$.
\end{thm}

\begin{proof}
We will prove the assertion in five steps.

\smallskip

\underline{Step 1}: \,Since $\overline{\Omega^\prime}$ is compact, it is covered by finitely many balls $\{B_{\rho}(x_j)\}_{j=1}^N\,\,(x_j \in \Omega^\prime,\,j=1,2,\ldots,N)$, where $N=N(\Omega^\prime)$, such that
\begin{equation*}
\overline{\Omega^\prime} = \bigcup_{j=1}^N \overline{B_{\rho}}(x_j), \quad  \rho <|x_i- x_{i+1}|<2\rho,\, \,\,B_{4\rho}(x_{i}) \subset \Omega,\,\,\textrm{for all}\,\,  1 \leq i \leq N,
\end{equation*}
where we put $x_{N+1}=x_1$.
For brevity we denote $B_{\rho}(x_j)$ by $B_j$ for each $j=1,2,\ldots, N$ and let $\sigma B_{j}:=B_{\sigma \rho}(x_{j})$ for $\sigma=2$ and 4.
\\[-5mm]

\begin{center}
\begin{tikzpicture}[scale=0.50]
            \draw  plot[draw=black, smooth, tension=.8] coordinates {(-4.0,0.55) (-3.3,2.8) (-1.2,4.0) (1.8,3.6) (4.4,3.8) (5.2,2.8) (5.3,0.6) (2.8,-2.5) (0,-1.0) (-3.3,-1.5) (-4.0,0.6)};
            \filldraw[pattern=north east lines, pattern color=blue, thick]  plot[smooth, tension=.8] coordinates {(-3.5,0.5) (-3,2.5) (-1,3.5) (1.5,3) (4,3.5) (5,2.5) (5,0.5) (2.5,-2) (0,-0.5) (-3,-1.0) (-3.5,0.5)};
            \draw plot[smooth, tension=.7] coordinates {(-3.5,0.5) (-3,2.5) (-1,3.5) (1.5,3) (4,3.5) (5,2.5) (5,0.5) (2.5,-2) (0,-0.5) (-3,-1.0) (-3.5,0.5)};
           \draw (1,1.5)node[below]{\small $\Omega'$};
           \draw (5.7,1.5)node[right]{\small $\Omega$};
           \draw (-1.9,2.1)node[right]{\footnotesize $4B_{m}$};
           \draw[thin] (-2.71,2.0) circle (26pt);
            \draw (-1.5,0.1)node[left]{\footnotesize $B_{j}$};
            \draw[thick] (-3.2,-0.3) circle (7pt);
            \draw[thick] (-3.25,0.1) circle (7pt);
            \draw[thick] (-3.2,0.5) circle (7pt);
            \draw[thick] (-3.18,0.9) circle (7pt);
            \draw[thick] (-3.17,1.3) circle (7pt);
             \draw[thick] (-3.05,1.7) circle (7pt);
             \draw[thick] (-2.9,2.1) circle (7pt);
            \draw[thin] (14,0) circle (80pt);
            \draw[thick] (14,0) circle (40pt);
            \draw[thick] (17,1.2) circle (40pt);
            \draw[thin] (17,1.2) circle (80pt);
            \draw (13,3.4)node[right]{\footnotesize $B_j$};
            \draw (16.5,-0.7)node[right]{\footnotesize $\frac{1}{2}B_{j+1}$};
            \draw (19.5,2.3)node[right]{\footnotesize $B_{j+1}$};
            \draw (14,0)node[below]{\footnotesize $x_j$};
            \draw (17.1,1.2)node[above]{\footnotesize $x_{j+1}$};
            \filldraw[thick] (14,0) circle (2pt);
            \filldraw[thick] (17,1.2) circle (2pt);
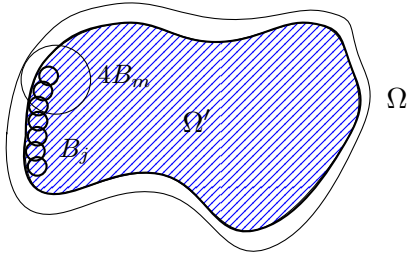
            \draw (0.2,-2.8)node[below] {\footnotesize  Figure:  Harnack's chain argument};
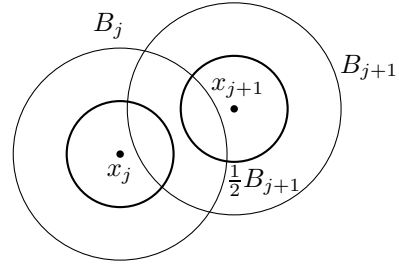
            \draw (16.2,-2.8)node[below] {\footnotesize  Figure:  Relations of two balls};
 \end{tikzpicture}
 \end{center}

By (\ref{ex. propassumption prime}), there exists at least one $B_j=B_\rho (x_j)$, denoted by  $x_1=x_j$ and $B_1=B_j$, such that
\begin{equation*}
|B_1 \cap \{u(t_0) \geq L\}| \geq \frac{\alpha}{2^n} |B_1|,
\end{equation*}
Thus, by Proposition \ref{ex. prop}, there exists positive numbers $\delta_0, \varepsilon_0 \in (0,1)$ depending only on $p, n$ and $\alpha_{0}=\alpha$ and independent of $L$ such that
\begin{equation*}
\big|B_1 \cap \{u(t) \geq \varepsilon_0L\}\big| \geq \frac{\,\alpha_{0}\,}{2^{n+1}} \big|B_1\big|.
\end{equation*}
for all $t \in [t_0,\,t_0+\delta_0L^{q+1-p}\rho^p]$. Let $Q_{4\rho}(z_1):=4B_1 \times (t_0,\,t_0+\delta_0L^{q+1-p}\rho^p) \subset \Omega^\prime_T$. By Lemma \ref{ex. crucial lemma}, for any $\nu_0 \in (0,1)$, there exists a positive number $\varepsilon_{\nu_0}$ depending only on $p,\,n,\,\alpha_{0},\,\delta,\,L,\nu_0,\,\varepsilon_0$ such that
\begin{equation}\label{def of Q(z_1)}
|Q_{4\rho}(z_1) \cap \{u<\varepsilon_{\nu_0}L\}| <\nu_0 |Q_{4\rho}(z_1)|.
\end{equation}
Here $\varepsilon_{\nu_0}:=\varepsilon_0/2^{J_0}$, where $J_0$ is determined in (\ref{def of J}) replaced $\nu$ by $\nu_0$. In particular, as noted in (\ref{choice of epsilon}), we choose $\varepsilon_0$ as $\varepsilon_0=\left(\frac{\delta_0}{2^{I_0}} \right)^\frac{1}{q+1-p}$, where $I_0$ is sufficiently large positive integer. Following the same argument as before, we next divide $Q_{4\rho}(z_1)$ along time direction into finitely many subcylinders of number $s_0:=2^{J_0+I_0}$, with each time-length $$k_{J_0}^{q+1-p}\rho^{p}=\frac{\delta_0L^{q+1-p}}{2^{J_0+I_0}}\rho^p$$ via (\ref{choice of k_j}) and put
\begin{equation*}
Q^{(\ell)}(z_1):=4B_1 \times \left(t_0+\ell\, \frac{\delta_0L^{q+1-p}}{2^{J_0+I_0}}\rho^p,\,t_0+(\ell+1)\frac{\delta_0L^{q+1-p}}{2^{J_0+I_0}}\rho^p \right)
\end{equation*}
for $\ell=0,1,\ldots, s_0-1$. Thus, by (\ref{def of Q(z_1)}), there exists a $Q^{(\ell)}(z_1)$ which holds that
\begin{equation*}
|Q^{(\ell)}(z_1) \cap \{u<\varepsilon_{\nu_0}L\}| <\nu_0 |Q^{(\ell)}(z_1)|.
\end{equation*}
Therefore it follows from Theorem \ref{Expansion of positivity} that
\begin{equation}\label{ex. positivity in Omega^prime eq1}
u \geq \eta_1L \quad \textrm{a.e.}\,\,\textrm{in}\,\,\,2B_1 \times \left(t_0+\left(\ell+\frac{1}{2}\right)\frac{\delta_0L^{q+1-p}}{2^{J_0+I_0}}\rho^p,\,t_0+(\ell+1)\frac{\delta_0L^{q+1-p}}{2^{J_0+I_0}}\rho^p \right),
\end{equation}
where $\displaystyle \eta_1:=\frac{1}{2L} k_{J_0}$.

\smallskip

\underline{Step 2}: \,  By $\rho<|x_1-x_2|<2\rho$,
\begin{equation*}
D_1:=B_1 \cap B_2 \neq \varnothing
\end{equation*}
holds.
\begin{center}
\begin{tikzpicture}[scale=0.7]
\filldraw[pattern=north east lines, pattern color=blue,thick] (0.9,-1.65) arc(-56:56:2);
\filldraw[pattern=north east lines, pattern color=blue,thick] (0.9,1.65) arc(124:236:2);
\draw[thick](-0.2,0) circle(2) (2,0) circle (2);
 \draw (0.4,0)node[right]{\small $D_1$};
\draw (-0.7,1.5)node[right]{\small $B_1$};
  \draw (2.2,1.4)node[right]{\small $B_2$};
   \draw (-0.3,0)node[below]{\small $x_1$};
  \draw (2.1,0)node[below]{\small $x_2$};
    \filldraw[thick] (-0.2,0) circle (2pt);
  \filldraw[thick] (2,0) circle (2pt);
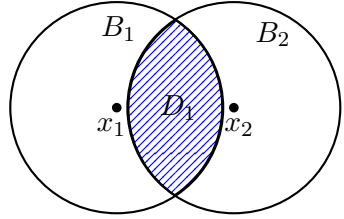
  \draw (0.8,-2.8)node[below] {\footnotesize  Figure:  Intersection of two balls};
\end{tikzpicture}
\end{center}
Via (\ref{ex. positivity in Omega^prime eq1}), we have
\begin{equation}\label{ex. positivity in Omega^prime eq2}
u \geq \eta_1L \quad \textrm{a.e.}\,\,D_1 \times \mathcal{I}_0^{(\ell)},
\end{equation}
where let $\mathcal{I}^{(\ell)}_0:=\left(t_0+\left(\ell+\frac{1}{2}\right)\frac{\delta_0L^{q+1-p}}{2^{J_0+I_0}}\rho^p,\,t_0+(\ell+1)\frac{\delta_0L^{q+1-p}}{2^{J_0+I_0}}\rho^p \right)$. Let $t_1$ be a middle point in $\mathcal{I}^{(\ell)}_0$ and, by (\ref{ex. positivity in Omega^prime eq2}), we observe that
\begin{equation*}
|D_1 \cap \{u(t_1) \geq \eta_1 L\}|=|D_1|,
\end{equation*}
which is, setting $\displaystyle \alpha_1:=\frac{|D_1|}{|B_2|} \in (0,1)$,
\begin{equation*}
\big|B_2 \cap \{u(t_1) \geq \eta_1 L\}\big| \geq \alpha_1 \big|B_2\big|.
\end{equation*}
By the very same argument as Step 1, there exist positive numbers $\delta_1\in (0,1)$, $I_{1}, J_{1} \in \mathbb{N}$ depending only on $p,\,n$ and $\alpha_1$ and independent of $L$, and $\nu_1 \in (0,1)$ such that
\begin{equation}\label{ex. positivity in Omega^prime eq4}
u \geq \eta_2L \quad \textrm{a.e.}\,\,\textrm{in}\,\,\,B_2 \times \left(t_1+\left(k+\frac{1}{2}\right)\frac{\delta_1(\eta_{1}L)^{q+1-p}}{2^{J_1+I_1}}\rho^p,\,t_1+(k+1)\frac{\delta_1(\eta_{1}L)^{q+1-p}}{2^{J_1+I_1}}\rho^p \right),
\end{equation}
for some
$$k \in \{0,1,\ldots, 2^{I_{1}+J_{1}}-1\},$$
where $\displaystyle \eta_2:=\frac{1}{2L}k_{J_{1}}$ and $J_1$ is chosen that $\displaystyle J_1 \geq \max\left\{ \left(\frac{C}{\nu_1\alpha_1\delta_1^{\frac{1}{p}}} \right)^\frac{p}{p-1},\,J_0 \right\}$.

\smallskip

\underline{Step 3}:\, We will proceed by induction on $m$. Assume that for some $m \in \{1,2.\ldots ,N\}$
\begin{equation}\label{ex. positivity in Omega^prime eq5}
u \geq \eta_mL \quad \textrm{a.e.}\,\,\textrm{in}\,\,B_m\times \mathcal{I}^{(\ell)}_{m-1}.
\end{equation}
Here we let
\begin{equation*}
\mathcal{I}^{(\ell)}_{m-1}:=\left(t_{m-1}+\left(\ell+\frac{1}{2}\right)\frac{\delta_{m-1}(\eta_{m-1}L)^{q+1-p}}{2^{J_{m-1}+I_{m-1}}}\rho^p,\,t_{m-1}+(\ell+1)\frac{\delta_{m-1}(\eta_{m-1}L)^{q+1-p}}{2^{J_{m-1}+I_{m-1}}}\rho^p \right),
\end{equation*}
where $\delta_{m-1},\,\eta_{m-1} \in (0,1)$ and $J_{m-1}, I_{m-1} \in \mathbb{N}$ are determined inductively, and $\ell \in \left\{0,1,\ldots, 2^{I_{m-1}+J_{m-1}}-1\right\}$. By $\rho <|x_m-x_{m+1}|<2\rho$ again,
\begin{equation*}
D_m:=B_m \cap B_{m+1}\neq \varnothing
\end{equation*}
and thus, (\ref{ex. positivity in Omega^prime eq5}) yields that
\begin{equation}\label{ex. positivity in Omega^prime eq6}
u \geq \eta_mL \quad \textrm{a.e.}\,\,\textrm{in}\,\,D_m \times \mathcal{I}^{(\ell)}_{m-1}.
\end{equation}
Letting $t_m$ be a middle point of $\mathcal{I}^{(\ell)}_{m-1}$ again and, by (\ref{ex. positivity in Omega^prime eq6}), we obtain that
\begin{equation*}
\big|B_{m+1} \cap \{u(t_m) \geq \eta_mL\}\big| \geq \alpha_m\big|B_{m+1}\big|,
\end{equation*}
where we let $\displaystyle \alpha_m:=\frac{|D_m|}{|B_{m+1}|} \in (0,1)$. Again, using the very same argument as in Step 1 and 2, there exist positive numbers $\delta_m,\,\eta_m \in (0,1)$, $J_{m},\,I_{m} \in \mathbb{N}$ depending only on $p,\,n$ and $\alpha_m$ and independent of $L$ and some $\nu_m \in (0,1)$ such that
\begin{equation}\label{ex. positivity in Omega^prime eq7}
u \geq \eta_{m+1}L \quad \textrm{a.e.}\,\,\textrm{in}\,\,B_{m+1}\times \mathcal{I}^{(k)}_m,
\end{equation}
where
\begin{equation*}\mathcal{I}^{(k)}_m:=\left(t_{m}+\left(k+\frac{1}{2}\right)\,\frac{\delta_{m}(\eta_{m}L)^{q+1-p}}{2^{J_{m}+I_{m}}}\rho^p,\,t_{m}+(k+1)\frac{\delta_{m}(\eta_{m}L)^{q+1-p}}{2^{J_{m}+I_{m}}}\rho^p \right)
\end{equation*}
for some
\begin{equation*}
k \in \{0,1,\ldots,2^{I_{m}+J_{m}}-1\},
\end{equation*}
where
\begin{equation}\label{def of eta_{m}}
\eta_{m+1}=\frac{1}{2L}k_{J_{m}}=\left(\frac{\delta_{m}}{2^{I_{m}+J_{m}}} \right)^{\frac{1}{q+1-p}}\eta_{m}
\end{equation}
and $J_m$ is chosen that $\displaystyle J_m \geq \max\left\{ \left(\frac{C}{\nu_m\alpha_m\delta_m^{\frac{1}{p}}} \right)^\frac{p}{p-1},\,J_{m-1} \right\}$.\\[2mm]
\noindent
\underline{Step 4}:\, We next claim the following: \\
\quad For any $k \in \left\{0,1,\ldots,2^{I_{m}+J_{m}}-1\right\}$ and $\ell \in \left\{0,1,\ldots, 2^{I_{m-1}+J_{m-1}}-1\right\}$,
\begin{equation}\label{ex. positivity in Omega^prime eq7}
\mathcal{I}^{(\ell)}_{m-1} \supset \mathcal{I}^{(k)}_m.
\end{equation}
\begin{center}
\begin{tikzpicture}[scale=1.1]
\draw[thin, ->] (-3.5,0)--(4,0) node[right] {\small $t$}; 
\draw[thick] (0.5,0) circle (0.07);
\filldraw[thick] (0,0) circle (0.07);
\draw node[below]{\small $t_{m}$};
\draw [thick](2.5,0) circle (0.07);
\coordinate (A) at (-3,0);
    \coordinate (B) at (3,0);
    \draw (A) -- (B);
    \draw [bend left,distance=2.3cm] (A)
         to node [fill=white,inner sep=0.2pt,circle] {\footnotesize $\mathcal{I}^{(\ell)}_{m-1}$} (B);
         \draw[thick] (-3,0) circle (0.07);
\draw [thick](3,0) circle (0.07);
\filldraw[very thick](0.5, 0)--(2.5,0);
\coordinate (C) at (0.5,0);
    \coordinate (D) at (2.5,0);
    \draw (C) -- (D);
    \draw [bend right,distance=1.1cm] (C)
         to node [fill=white,inner sep=0.2pt,circle] {\footnotesize $\mathcal{I}^{(k)}_{m}$} (D);
        \draw (D)node[below]{\small $b$};
        \draw (B)node[below]{\small $a$};
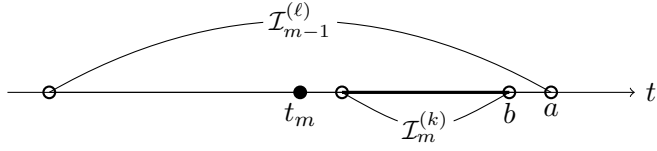
        \draw (0.2,-0.8)node[below] {\footnotesize  Figure:  Relations of $\mathcal{I}^{(\ell)}_{m-1}$ and $\mathcal{I}^{(\ell)}_{m}$};
\end{tikzpicture}
\end{center}

\noindent
Since this \eqref{ex. positivity in Omega^prime eq7} is equivalent to
\begin{equation}\label{ex. positivity in Omega^prime eq8}
a=t_{m-1}+(\ell+1) \frac{\delta_{m-1}(\eta_{m-1}L)^{q+1-p}}{2^{J_{m-1}+I_{m-1}}}\rho^p>t_{m}+(k+1) \frac{\delta_{m}(\eta_{m}L)^{q+1-p}}{2^{J_{m}+I_{m}}}\rho^p=b,
\end{equation}
we will verify (\ref{ex. positivity in Omega^prime eq8}). Using $\eta_{m}=\frac{1}{2L}k_{J_{m-1}}=\left(\frac{\delta_{m-1}}{2^{I_{m-1}+J_{m-1}}} \right)^{\frac{1}{q+1-p}}\eta_{m-1}$ via (\ref{def of eta_{m}}) and
\begin{equation*}
t_m=t_{m-1}+\left(\ell+\frac{3}{4}\right)\frac{\delta_{m-1}(\eta_{m-1}L)^{q+1-p}}{2^{J_{m-1}+I_{m-1}}}\rho^p :\quad \textrm{middle point of}\,\,\mathcal{I}^{(\ell)}_{m-1},
\end{equation*}
we find that $a-b$ is estimated as follows:
\begin{align}
a-b&=\bigg[\frac{1}{4}-(k+1)\frac{\delta_{m}}{2^{J_{m}+I_{m}}} \bigg] \frac{\delta_{m-1}(\eta_{m-1}L)^{q+1-p}}{2^{J_{m-1}+I_{m-1}}}\rho^p \notag \\
&>\bigg[\frac{1}{4}-2^{J_m+I_m}\cdot \frac{\delta_{m}}{2^{J_{m}+I_{m}}} \bigg]\frac{\delta_{m-1}(\eta_{m-1}L)^{q+1-p}}{2^{J_{m-1}+I_{m-1}}}\rho^p  \notag \\
&>\bigg[\frac{1}{4}-\delta_m\bigg] \frac{\delta_{m-1}(\eta_{m-1}L)^{q+1-p}}{2^{J_{m-1}+I_{m-1}}}\rho^p \notag \\
&>0 \notag
\end{align}
since $\delta_m=\frac{\alpha_m^{p+1}}{2^{3p+3}Cn^p}<\frac{1}{4}$ by (\ref{parameter}). Thus, (\ref{ex. positivity in Omega^prime eq7}) is actually valid. \\[2mm]
\underline{Step 5}:\, By Step 3, we have, for all $m=1,2,\ldots,N$,
\begin{equation*}
u \geq \eta_{m+1}L \quad \textrm{a.e.}\,\,\textrm{in}\,\,B_{m+1}\times \mathcal{I}^{(k)}_m,
\end{equation*}
where we let $B_{N+1}:=B_1$. Since $\{\eta_m\}_{m=1}^{N+1}$ is decreasing this inequality plainly yields that, for all $m=1,2,\ldots,N$,
\begin{equation}\label{ex. positivity in Omega^prime eq9}
u \geq \eta_{N+1}L \quad \textrm{a.e.}\,\,\textrm{in}\,\,B_{m+1}\times \mathcal{I}^{(k)}_m.
\end{equation}
By (\ref{ex. positivity in Omega^prime eq7}) in Step 4, we furthermore find that
\begin{equation}\label{ex. positivity in Omega^prime eq10}
\mathcal{I}^{(\ell)}_1 \supset \cdots \supset \mathcal{I}^{(k)}_N
\end{equation}
where $\ell \in \left\{0,1,\ldots, 2^{J_0+I_0}-1\right\}$ and $k \in \left\{0,1,\ldots, 2^{J_N+I_N}-1\right\}$. From (\ref{ex. positivity in Omega^prime eq9}) and (\ref{ex. positivity in Omega^prime eq10}) it follows that
\begin{equation*}
u \geq \eta_{N+1}L\quad \textrm{a.e.}\,\,\textrm{in} \,\,\,\Omega^{\prime} \times \mathcal{I}^{(k)}_N
\end{equation*}
and thus, we arrive at the conclusion.
\end{proof}

As mentioned in Corollary \ref{key cor}, if a solution of (\ref{maineq'}) is positive almost everywhere in $\Omega'$ at some time $t_0$, its positivity expands in space-time without ''waiting time''.

\begin{cor}\label{ex. cor prime}
Let $u$ be a nonnegative weak supersolution of (\ref{maineq'}). Let $\Omega^\prime$ be a subdomain contained compactly in $\Omega$.  Suppose that $u(t_0)>0$ almost everywhere in $\Omega^{\prime}$ for some  $t_0 \in [0,T)$. Then there exist positive numbers $\eta_0$ and  $\tau_0$ such that
\begin{equation*}
u \geq \eta_0 \quad \textrm{a.e.}\quad  \textrm{in}\quad  \Omega^\prime \times (t_0, t_0+\tau_0).
\end{equation*}
\end{cor}
\begin{proof}
Since $\overline{\Omega^\prime}$ is compact, it is covered by finitely many balls $\{B_{\rho}(x_j)\}_{j=1}^N\,\,(x_j \in \Omega^\prime,\,j=1,2,\ldots,N)$, where $N=N(\Omega^\prime)$, such that
\begin{equation*}
\overline{\Omega^\prime} = \bigcup_{j=1}^N \overline{B_{\rho}}(x_j), \quad  \rho <|x_i- x_{i+1}|<2\rho,\, \,\,B_{2\rho}(x_{i}) \subset \Omega,\,\,\textrm{for all}\,\,  1 \leq i \leq N,
\end{equation*}
where we put $x_{N+1}=x_1$.
For brevity we denote $B_{\rho}(x_j)$ by $B_j$ for each $j=1,2,\ldots, N$ and let $2B_{j}:=B_{2\rho}(x_{j})$. By assumption, $u>0$ almost everywhere in each ball $2B_{j}$, $j=1,\ldots, N$. Corollary \ref{key cor} yield that there exist positive numbers $\eta_{1}$ and $\tau_{1}$ such that
\begin{equation*}
u \geq \eta_{1}\quad \textrm{a.e.}\,\,\textrm{in}\,\,B_{1}\times (t_{0},t_{0}+\tau_{1}).
\end{equation*}
From the Harnack chain argument as in the proof of Theorem \ref{ex. thm prime} it follows that there exists positive numbers $\tau_2<\tau_1,\,\eta_2 <\eta_1$ such that
\begin{equation*}
u \geq \eta_{2}\quad \textrm{a.e.}\,\,\textrm{in}\,\,B_{2}\times (t_{0},t_{0}+\tau_{2}).
\end{equation*}
Iterative this argument finitely, there exist $\eta_1>\eta_2>\cdots>\eta_{N}$ and $\tau_1>\tau_2>\cdots>\tau_{N}$ such that
\begin{equation*}
u \geq \eta_{N}\quad \textrm{a.e.}\,\,\textrm{in}\,\,B_{j}\times (t_{0},t_{0}+\tau_{N}).
\end{equation*}
for all $j=1,\ldots, N$. Thus, letting $\eta_0:=\eta_N$ and $\tau_0:=\tau_N$, we complete the proof.
\end{proof}

 \subsection{Positivity near the boundary}\label{Subsec. Positivity near the boundary}

We next study the positivity of the solution to the doubly nonlinear equations of $p$-Sobolev flow type (\ref{maineq'}) near the boundary. In what follows, assume that the bounded domain $\Omega$ satisfies the \textit{interior ball condition}, that is, for every boundary point $\xi \in \partial \Omega$, there exist a point $x_0 \in \Omega$ and some $\rho>0$ such that
\begin{equation*}
\overline{B_\rho(x_0)} \cap \partial \Omega=\{\xi\},
\end{equation*}
where $\overline{B_\rho(x_0)}$ denotes the closure of $B_\rho(x_0)$.

\begin{prop}[Positivity of the solution near the boundary]\label{Positivity near the boundary}
Assume $u_0>0$ in $\Omega$. Then every nonnegative supersolution $u$ to
(\ref{maineq'}) is positive near the boundary.
\end{prop}

\begin{proof}
We will follow the similar idea as \cite{Suzuki-Ueoka}.  Since, $\Omega$ satisfies the interior ball condition, we have, for every boundary point $\xi \in \partial \Omega$,
\begin{equation*}
\overline{B_\rho(x_0)} \cap \partial \Omega=\{\xi\}.
\end{equation*}
Take $\rho' \in (0,\rho)$ and define the annulus
\begin{equation*}
A:=B_\rho(x_0)\, \backslash \,\overline{B_{\rho'}(x_0)} .
\end{equation*}
\begin{center}
\begin{tikzpicture}[samples=70, rotate=-30, thick,scale=1.1]
\draw [very thick,domain=-2.2:2.2,] plot(\x, {exp(-0.33 * \x * \x)}); 
\draw (1.5,0.7)node[above] {\small $\partial \Omega$};
\draw[->,thin] (0,0)-- (0,2);
\draw [pattern=north east lines, pattern color=blue,thick](0,0) circle (0.97);
\draw[pattern=north east lines, pattern color=white,thick](0,0) circle (0.49);
\filldraw (0,1) circle (0.05);
\draw (-0.2,1) node[above]{\small $\xi$};
\filldraw (0,0) circle (0.04);
\draw (0,0) node[below]{\footnotesize $x_{0}$};
\draw (0.3,-0.3) node[below]{\footnotesize $A$};
\draw (0,2) node[right]{\footnotesize$r=|x-x_{0}|$};
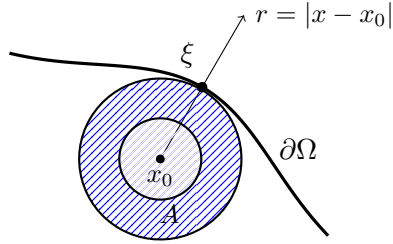
\draw (0.6,-1.3)node[below] {\footnotesize  Figure:  Interior ball condition};
\end{tikzpicture}
\qquad \qquad
\begin{tikzpicture}[scale=1.2]
\draw [->,thin](-0.2,0)--(3,0);
\draw [->, thin] (0,-0.5)--(0,1.6);
\draw [thick,domain=0:2.5,] plot(\x, {exp(-0.5* \x * \x)-exp(-0.5* 2 * 2)}); 
\draw (2,0.8) node[above]{\footnotesize$v(x)=e^{-\alpha r^2}-e^{-\alpha\rho^2}$};
\draw (0,-0.2) node[left]{\footnotesize $x_{0}$};
\filldraw (2,0) circle (0.05);
\draw (2,0) node[below]{\footnotesize $\xi$};
\filldraw (1,0) circle (0.05);
\draw [pattern=north east lines, pattern color=blue,thin] (1,-0.05) rectangle (2,0.05);
\coordinate (A) at (1,0);
    \coordinate (B) at (2,0) ;
    \draw (A) -- (B);
    \draw [bend right,distance=0.6cm, thin] (A)
         to node [fill=white,inner sep=0.2pt,circle] {\footnotesize $A$} (B);
 \draw (3,0) node[below]{\footnotesize$r$};
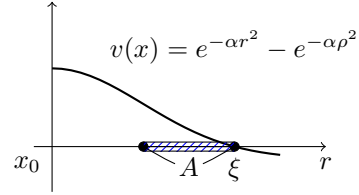
 \draw (0.8,-0.6)node[below] {\footnotesize  Figure:  Graph of $v(x)$};
  \end{tikzpicture}
\end{center}

We define a function $v$ for $(x, t) \in A \times [0, T]$ as
\begin{equation}\label{positivity bdry eq0}
v(x,t)=v(x):=e^{-\alpha r^2}-e^{-\alpha\rho^2},
\end{equation}
where $r:=|x-x_0|<\rho$ and $\alpha>0$ is to be determined later. Since
\begin{equation*}
\begin{cases}
v_{x_j}=-2\alpha(x_j-x_{0,j})e^{-\alpha r^2}, \quad &j=1,\ldots, n\\
v_{x_ix_j}=-2\alpha \delta_{ij}e^{-\alpha r^2}+4\alpha^2 (x_i-x_{0,i})(x_j-x_{0,j})e^{-\alpha r^2}, \quad & i,j=1,\ldots, n,
\end{cases}
\end{equation*}
we have
\begin{align}
\Delta_pv&=(p-2)|\nabla v|^{p-4}e^{-3\alpha r^2}8 \alpha^3 r^2 (- 1 + 2 \alpha r) \notag \\
&\qquad \qquad \qquad +|\nabla v|^{p-2}e^{-\alpha r^2} 2 \alpha (- n + 2 \alpha r^2)
\end{align}
and thus, we can choose a sufficiently large $\alpha$ so that
\begin{equation*}\Delta_p v(x,t) \geq 0\quad \mathrm{in} \,\,A \times [0,T],
\end{equation*}
where $\alpha$ is chosen depending on $\rho$ and $\rho^\prime$.
Therefore, by $\partial_t v^q=0$
\begin{equation*}
\partial_t v^q-\Delta_p v \leq 0 \quad \textrm{in}\quad  A \times [0,T].
\end{equation*}
Now, let $m:=\min\left\{\min \limits_\Omega u_0,\,\min \limits_{\partial B_{\rho^\prime} (x_0)
\times [0, T]} u \right\}$. We will show that $mv (x,t)$ is a lower comparison function for the solution. We note that the solution $u$ is uniformly (H\"older) continuous in $\Omega_T =\Omega \times (0,T)$ (see Section \ref{Holder regularity} below), again, we can choose $\alpha> 0$ to be so large that,
on the initial boundary $A \times \{t=0\}$,
\begin{equation}\label{positivity bdry eq1}
u(x,0)=u_0 (x) \geq mv (x, 0)
=m (e^{-\alpha r^2} - e^{-\alpha \rho^2}).
\end{equation}
Also,
\begin{equation}\label{positivity bdry eq2}
u(x,t) \geq mv(x,t)\quad \mathrm{on}\,\,\partial B_\rho(x_0) \times (0,T),
\end{equation}
 since, on $\partial B_\rho (x_0) \times [0,T]$, $v=0$ by definition and $u\geq 0$ by Corollary \ref{key cor}. On $\partial B_{\rho^\prime} (x_0) \times [0,T]$, by the very definition of $m$,
\begin{equation}\label{positivity bdry eq3}
u(x,t) \geq m \big(e^{-\alpha \rho'^2}-e^{-\alpha\rho^2}\big)=mv(x,t)\quad \mathrm{on} \,\,\partial B_{\rho'}(x_0)\times (0,T)
\end{equation}
From (\ref{positivity bdry eq1}), (\ref{positivity bdry eq2}) and (\ref{positivity bdry eq3}), we find that
\begin{equation*}
u \geq mv \quad \textrm{on}\,\,\partial_pA_T,
\end{equation*}
where $A_T:=A\times (0,T)$ and $\partial_p A_T$ is the parabolic boundary of $A_T$ and thus, we have that $mv(x,t)$ is lower comparison function for $u$ in $A_T=A \times (0,T)$. By Theorem \ref{Comparison theorem}, we arrive at
\begin{equation}\label{positivity bdry eq4}
u \geq mv>0\quad \textrm{in}\,\, A_T.
\end{equation}
Thus the assertion is actually verified.
 \end{proof}

\section{The $p$-Sobolev flow}\label{Sect. The p-Sobolev flow}
\subsection{Positivity of the $p$-Sobolev flow}

In what follows, we consider the $p$-Sobolev flow (\ref{pS}). We first notice the nonnegativity of a solution of the $p$-Sobolev flow and its proof.
%
%


\begin{prop}[Nonnegativity of the $p$-Sobolev flow]\label{Nonnegativity prop pS}
Suppose $u_{0} \geq 0$ in $\Omega$. Then, a weak solution $u$ of (\ref{pS}) satisfies
\begin{equation}\label{nonnegativity pS}
u \geq 0 \quad \textrm{in}\quad \Omega_T.
\end{equation}
\end{prop}

\begin{proof}
Let  $0 < t_1 < t \leq T$ be arbitrarily taken and fixed.
Let $\sigma_{t_{1},t}$ be the same Lipschitz cut-off function on time as in the proof of Proposition \ref{Nonnegativity prop}. The function $-(- u)_+ \sigma_{t_1, t}$ is an admissible test function in (D2) of Definition \ref{def of weak sol.}, since $\partial_t (|u|^{q - 1} u) \in L^2 (\Omega_T)$ by (D1) of Defintion \ref{def of weak sol.} and,
$-(- u)_+ \sigma_{t_1, t}$ is in $L^{q+1} (\Omega \times (t_1, t))$. Thus, we have
\begin{align}\label{eq1 of nonnegativity pS}
\int_{\Omega_{t_1, t}}\partial_{\tau}(|u|^{q-1}(-u))(-u)_{+}\sigma_{t_1, t}\,dz&+\int_{\Omega_{t_1, t}}|\nabla u|^{p-2}\nabla (-u)\cdot \nabla\left((-u)_{+}\sigma_{t_1, t} \right)\,dz \notag \\
&= \int_{\Omega_{t_1, t}}\lambda(\tau)|u|^{q-1}(-u)(-u)_{+}\sigma_{t_1, t}\,dz.
\end{align}
Applying the very same argument as in the proof of Proposition \ref{Nonnegativity prop} to \eqref{eq1 of nonnegativity pS}, we obtain that
\[
\frac{q}{q+1}\int_{\Omega}(-u(t))_{+}^{q+1}\,dx \leq \int_0^t(\lambda(\tau))_{+}\int_{\Omega}(-u(\tau))_+^{q+1}\,dxd\tau.
\]
From the Gronwall lemma it follows that
\[
 \int_{\Omega}(-u(t))_{+}^{q+1}\,dx \leq 0
\]
since (D4) in Definition \ref{def of weak sol.}, $(-u(t))_+ \to 0$ in $L^{q+1} (\Omega)$ as $t \searrow 0$. Therefore we have $-u(x,t) \leq 0$ \,for $(x,t) \in \Omega_{T}$ and the claim is actually verified.
\end{proof}

We now state the fundamental energy estimate.

\begin{prop}[Energy equality]\label{Energy equality pS}
Let $u$ be a nonnegative solution to (\ref{pS}). Then the following identities are valid:
\begin{enumerate}[(i)]
\item
\begin{equation*}
\lambda(t)=\int_{\Omega}|\nabla u(x,t)|^{p}\,dx,\quad t \in [0,T] ;
\end{equation*}
\item
\begin{equation*}
q\int_{\Omega_{0,t}}u^{q-1}(\partial_t u)^2\,dz+\frac{1}{p}\lambda(t)=\frac{1}{p}\lambda(0),\quad t \in [0,T].
\end{equation*}
In particular,
\begin{equation}\label{lambdaineq}
\lambda(t) \leq \lambda(0),\quad t \in [0,T].
\end{equation}
\end{enumerate}
%
%
%
\end{prop}
The proof of this proposition is postponed, and will be given in Appendix \ref{Sec. Proof energy equality}.

\smallskip

%

%

\begin{prop}[Boundedness of the $p$-Sobolev flow]\label{Boundedness of the p-Sobolev flow}
Let $u \geq 0$ be a weak solution of the $p$-Sobolev flow equation (\ref{pS}). Then $u$ is bounded from above in $\Omega_T$ and
\begin{equation*} 
\| (u(t))_+ \|_{L^\infty(\Omega)} \leq e^{\lambda(0) T/q}\|u_{0}\|_{L^{\infty}(\Omega)}.
\end{equation*}

\end{prop}

\begin{proof} By~\eqref{lambdaineq} we have that $\lambda(t) \leq \lambda(0)$. Therefore, $u$ is a weak subsolution of~\eqref{maineq} with $M=\|u_{0}\|_{L^{\infty}(\Omega)}$ and $c = \lambda(0)$. The result then follows by Proposition~\ref{Boundedness}.
\end{proof}

In general, the solution to (\ref{maineq'}) may vanish at a finite time, however, under the volume constraint as in (\ref{pS}), the solution may positively expand in all of times (see Corollary \ref{ex. cor prime}). This is actually the assertion of the following proposition.

\begin{prop}[Interior positivity by the volume constraint]\label{Interior positivity by the volume constraint}
Let $\Omega'$ be a subdomain compactly contained in $\Omega$
and very close to $\Omega$.  Let $T$ be any positive number and assume that $u_0$ is continuous and $u_0>0$ in $\Omega$. Let $u$ be a nonnegative weak solution of (\ref{pS}). Then there exists a positive constant $\bar{\eta}$ such that %
\[
u(x,t)\geq \bar{\eta} \quad \textrm{in}\quad \Omega^{\prime} \times [0, T].
\]

\end{prop}
\begin{proof}[Proof of Proposition \ref{Interior positivity by the volume constraint}]
By the volume constraint together with Proposition~\ref{Boundedness of the p-Sobolev flow}, letting $M:=e^{\lambda(0)T/q}\|u_0\|_{L^\infty(\Omega)}$, we have, for a positive number $L<M$ and any $t \in [0,T]$
\begin{align}
1=\int_\Omega u^{q+1}(t)\,dx&=\int_{\Omega' \cap \{u(t) \geq L\} }u^{q+1}(t)\,dx+\int_{\Omega' \cap \{u(t) < L\} } u^{q+1}(t)\,dx+\int_{\Omega\,\backslash \Omega'} u^{q+1}(t)\,dx \notag \\
&\leq M^{q+1}\big|\Omega' \cap \{u(t) \geq L\} \big|+L^{q+1}|\Omega'|+M^{q+1}|\Omega\,\backslash \Omega'|\,;\, \notag
\end{align}
i.e.,
\[
\frac{1-L^{q+1}|\Omega'|-M^{q+1}|\Omega\,\backslash \Omega'|}{M^{q+1}} \leq \big|\Omega' \cap \{u(t) \geq L\} \big|.
\]
Choose $\Omega^{\prime}$ such that $|\Omega\,\backslash \Omega'| \leq \frac{1}{4M^{q+1}}$ and $L>0$ satisfying $L^{q+1}|\Omega'|<\frac{1}{4}$. Under such choice of $\Omega^\prime$ and $L$, we find that, for any $t \in (0,T]$,
\begin{equation}\label{positivity with volume constraint eq1}
\alpha|\Omega'| \leq \big|\Omega' \cap \{u(t) \geq L\} \big|,
\end{equation}
where $\displaystyle \alpha:=\frac{1}{2 M^{q+1}|\Omega'|}$. Using \eqref{lambdaineq}, for a nonnegative weak solution $u$ of \eqref{pS}, we see that $u$ is a weak supersolution to \eqref{maineq'} with $c=\lambda(T)$.  Thus, from Proposition \ref{ex. thm prime}, there exist positive integer $N=N(\Omega^\prime)$ and positive number families  $\{\delta_m\}_{m=0}^N,\,\{\eta_m\}_{m=1}^{N+1}\subset (0,1)$, $\{J_m\}_{m=0}^N,\,\{I_m\}_{m=0}^N \subset \mathbb{N}$ depending on $p, n, \alpha$ and independent of $L$, a time $t_N>t$ such that, for any $t \in [0,T]$,
\[
u \geq \eta_{N+1} L\quad \textrm{a.e.}\,\,\textrm{in}\,\,\Omega^{\prime}\times \mathcal{I}^{(k)}_N(t)
\]
where $\mathcal{I}^{(k)}_N(t):=\left(t_N+\left(k+\frac{1}{2}\right)\frac{\delta_N(\eta_NL)^{q+1-p}}{2^{J_N+I_N}}\rho^p,\,\,t_N+(k+1)\frac{\delta_N(\eta_NL)^{q+1-p}}{2^{J_N+I_N}}\rho^p \right)$ for some $k \in \left\{0,1,\ldots, 2^{J_N+I_N}-1\right\}$, and $t_{N}$ is written as
\[
t_{N}=t+\sum \limits_{m=1}^{N}\left(\ell+\frac{3}{4}\right)\frac{\delta_{m-1}(\eta_{m-1}L)^{q+1-p}}{2^{J_{m-1}+I_{m-1}}}\rho^p
\]
for some $\ell \in \{0,1,\ldots,2^{J_{m-1}+I_{m-1}}-1\}$. On the other hand, from $u_0>0$ in $\Omega$ and  Corollary \ref{ex. cor prime}, there exist positive number $\eta$ and $\tau$ such that
\[
u \geq \eta \quad \textrm{a.e.}\,\,\textrm{in}\quad \Omega^\prime \times (0,\tau).
\]
Here we can choose that $\mathcal{I}^{(k)}_N(0) \subset (0,\tau)$ from the proof of Proposition \ref{ex. thm prime} and Corollary \ref{ex. cor prime}. Since $t \in [0, T]$ is any nonnegative time, letting $\bar{\eta}:=\min\{\eta,\eta_{N+1}L\}$, we have that
\[
u(x,t)>\bar{\eta}  \quad \textrm{a.e.}\,\,\textrm{in}\quad \Omega^{\prime } \times [0,T],
\]
which is our assertion of Proposition \ref{Interior positivity by the volume constraint}.
\end{proof}

\begin{prop}[Positivity around the boundary for $p$-Sobolev flow]\label{positivity around the boundary for p-Sobolev flow}
Suppose that $u_0>0$ in $\Omega$. Let $u$ be any nonnegative weak solution $u$ to (\ref{pS}). Then $u$ is positive near the boundary.
\end{prop}

\begin{proof}
Since a nonegative weak solution of (\ref{pS}) is a nonnegative weak supersolution of (\ref{maineq'}) in $\Omega_T$, for $T>0$ with $c = \lambda(T)$, we can apply the proof of Proposition \ref{Positivity near the boundary}. Thus the proof is complete.
\end{proof}

\subsection{H\"older and gradient H\"{o}lder continuity}\label{Holder regularity}
In this section, we will prove the H\"older and gradient H\"{o}lder continuity of the solution to $p$-Sobolev flow (\ref{pS}) with respect to space-time variable.
\medskip

Suppose $u_0 >0$ in $\Omega$. Then by Propositions \ref{Interior positivity by the volume constraint} and \ref{Boundedness of the p-Sobolev flow}, for any $\Omega^\prime$ compactly contained in $\Omega$ and $T \in (0,\infty)$,
we can choose a positive constant $\tilde{c}$ such that
\begin{equation}\label{bounded from above and below1}
0<\tilde{c} \leq u \leq M=:e^{\lambda(0)T/q}\|u_{0}\|_{L^{\infty}(\Omega)} \quad \textrm{in}\quad  \Omega^\prime \times [0, T].
\end{equation}
%
Under such positivity of a solution in the domain as in (\ref{bounded from above and below1}), we can rewrite the first equation of (\ref{pS}) as follows : Set $v:=u^q$, which is equivalent to $u=v^\frac{1}{q}$ and put $g:=\frac{1}{q} v^{1/q-1}$ and then, we find that the first equation of (\ref{pS}) is equivalent to
\begin{equation}\label{eq. of v}
\partial_tv-\mathrm{div} \big(|\nabla v|^{p-2}g^{p-1}\nabla v \big)=\lambda(t)v \quad \textrm{in}\quad \Omega^\prime \times [0,T]
\end{equation}
and thus, $v$ is a positive and bounded weak solution of the evolutionary $p$-Laplacian  equation (\ref{eq. of v}). By (\ref{bounded from above and below1}) $g$ is uniformly elliptic and bounded in $\Omega'_\infty$. Then we have a local energy inequality for a local weak solution $v$ to (\ref{eq. of v}) in Appendix \ref{subsect. energy of v} (see \cite{DiBenedetto1}). \\

The following H\"{o}lder continuity is proved via using the local energy inequality, Lemma \ref{energy of v} in Appendix \ref{subsect. energy of v} and standard iterative real analysis methods. See  \cite[Chapter III]{DiBenedetto1} or \cite[Section 4.4, pp.44--47]{Urbano} for more details.

\begin{thm}[H\"{o}lder continuity]\label{Holder continuity}
Let $v$ be a positive and bounded weak solution to (\ref{eq. of v}). Then $v$ is H\"{o}lder continuous in $\Omega^\prime_T$ with a H\"older exponent $\beta \in (0, 1)$ on a space-time metric $|x|+|t|^{1/p}$ for any $T \in (0,\infty)$.
\end{thm}

%
%
%

By a positivity and boundedness as in (\ref{bounded from above and below1}) and a H\"older continuity in Theorem \ref{Holder continuity}, we see that the coefficient $g^{p-1}$ is H\"older continuous and thus, obtain a H\"older continuity of its spacial gradient.

\begin{thm}[Gradient interior H\"{o}lder continuity]\label{Gradient Holder continuity}
Let $v$ be a positive and bounded weak solution to (\ref{eq. of v}). Then, there exist a positive exponent $\alpha<1$ depending only on $n,p,\beta$ and a positive constant $C$ depending only on $n, p, \tilde{c}, M, \lambda(0), \beta, \|\nabla v\|_{L^p(\Omega'_T)},  [g]_{\beta, \Omega'_T}$ and $[v]_{\beta, \Omega'_T}$ such that $\nabla v$ is H\"{o}lder continuous in $\Omega'_T$ with an exponent $\alpha$ on the usual parabolic distance. Furthermore, its H\"{o}lder constant is bounded above by $C$, where $[f]_\beta$ denote the H\"older semi-norm of a H\"older continuous function $f$ with a H\"older exponent~$\beta$.
\end{thm}

The outline of proof of Theorem \ref{Gradient Holder continuity} is presented in Appendix \ref{Sec. Notes on Holder regularity}.

\medskip
By an elementary algebraic estimate and a interior positivity, boundedness
and a H\"older regularity of $v$ and its gradient $\nabla v$ in Theorems \ref{Holder continuity} and \ref{Gradient Holder continuity},
we also have a H\"{o}lder regularity of the solution $u$ and its gradient $\nabla u$.

\begin{thm} [H\"{o}lder and Gradient H\"{o}lder continuity of solutions to the $p$-Sobolev flow]
Let $u$ be a positive and bounded weak solution to the $p$-Sobolev flow (\ref{pS}). Then, there exist a positive exponent $\gamma<1$ depending only on $n,p,\beta, \alpha$ and a positive constant $C$ depending only on $n, p, \tilde{c}, M, \lambda(0), \beta, \alpha,\|\nabla u\|_{L^p(\Omega'_T)},  [g]_{\beta, \Omega'_T}$ and $[v]_{\beta, \Omega'_T}$ such that $u$ and $\nabla u$ is H\"{o}lder continuous in $\Omega'_T$ with an exponent $\gamma$ on a parabolic metric $|x|+|t|^{1/p}$ and on the parabolic one, respectively. The H\"{o}lder constants are bounded above by $C$, where $[f]_\beta$ denote the H\"older semi-norm of a H\"older continuous function $f$ with a H\"{o}lder exponent $\beta$.

\end{thm}



\begin{appendices}
\section{Some fundamental facts}\label{Sec. Some fundamental facts}

\subsection{$L^2$ estimate of the time derivative}

We will show the existence in $L^2 (\Omega_T)$ of time-derivative for a weak solution to \eqref{pS}.

\begin{lem} \label{existence of u_t}
Let $u$ be a nonnegative solution to (\ref{pS}).Then there exists $\partial_tu$ in a weak sense, such that  $\partial_tu \in L^2(\Omega_T)$.
\end{lem}


\begin{proof}[Proof of Lemma \ref{existence of u_t}]
Let $a>0$ and $\varepsilon >0$ be arbitrary given. Let $u\geq 0$ be a weak solution to (\ref{pS}).  Let us define a truncated Lipschitz function $\phi_{\varepsilon}(x)$ by
\begin{equation*}
\phi_{\varepsilon}(x):=
\begin{cases}
0\quad &(0 \leq x \leq a)\\
\frac{1}{\varepsilon} (x-a)\quad &(a \leq x \leq a+\varepsilon)\\
1\quad &(x\geq a+\varepsilon).
\end{cases}
\end{equation*}
 We also set $h(v):=v^{1/q}$ for $v \geq 0$. For any $\varphi \in C^{\infty}_{0}(\Omega_{T})$,
 \begin{align}\label{eq.A1}
 \int_{\Omega_{T}}u \phi_{\varepsilon}(u)\partial_{t}\varphi\,dz&=-\int_{\Omega_{T}}\partial_{t}(u \phi_{\varepsilon} (u)) \varphi\,dz =-\int_{\Omega_{T}}\partial_{t}((u^{q})^{1/q} \phi_{\varepsilon} (u)) \varphi\,dz \notag \\
 &=-\int_{\Omega_{T}}[\partial_{t}h(u^{q})\phi_{\varepsilon}(u)+u \partial_{t}\phi_{\varepsilon}(u)]\varphi\,dz \notag \\
 &=-\int_{\Omega_{T}}\left[\partial_{t}h(u^{q})\phi_{\varepsilon}(u)+u \frac{\,1\,}{\varepsilon} \chi_{\{a \leq u \leq a+\varepsilon\}} \partial_{t}((u^{q})^{1/q} )\right]\,\varphi\,dz \notag \\
 &=-\int_{\Omega_{T}}\partial_{t}h(u^{q}) \left(\phi_{\varepsilon}(u)+u \frac{\,1\,}{\varepsilon} \chi_{\{a \leq u \leq a+\varepsilon\}} \right)\varphi\,dz.
 \end{align}
 We note that $h(v)=v^{1/q}$ is locally Lipschitz on $\{v = u^q : u \geq a\}$ and $\partial_{t}u^{q} \in L^{2}(\Omega_{T})$ by the very definition (D1) of Definition \ref{def of weak sol.} and thus, a composite function $h(u^{q})$ is weak differentiable in $\{u \geq a\}$ and
 \begin{equation*}
 \partial_{t}h(u^{q})=h'(u^{q})\cdot \partial_{t}u^{q} \in L^{2}(\Omega_{T} \cap \{u \geq a\}),
 \end{equation*}
 since $h'(u^{q})=\frac{1}{q}u^{1-q} \leq \frac{1}{q}a^{1-q}$ on $\{u \geq a\}$. Taking into account of
 \begin{equation*}
 \partial_{t}h(u^{q})\phi_{\varepsilon}(u) \to \partial_{t}h(u^{q})\chi_{\{u \geq a\}}\quad (\varepsilon \searrow 0),
 \end{equation*}
 \begin{equation*}
 |\partial_{t}h(u^{q})\phi_{\varepsilon}(u)| \leq \partial_{t}h(u^{q})\chi_{\{u\geq a\}} \in L^{1}(\Omega_{T})
 \end{equation*}
 and using the Lebesgue dominated convergence theorem, we have, for the first term on the right hand side of (\ref{eq.A1}),
 \begin{equation}\label{eq.A2}
 -\lim_{\varepsilon \searrow 0} \int_{\Omega_{T}}\partial_{t}h(u^{q})\phi_{\varepsilon}(u) \varphi\,dz=-\int_{\Omega_{T}}\partial_{t}h(u^{q})\chi_{\{u\geq a\}}\varphi\,dz.
 \end{equation}
 By Lemma \ref{conv. Dirac} the second term on the right hand side of (\ref{eq.A1}) is computed as
 \begin{equation}\label{eq.A3}
 \lim_{\varepsilon \searrow 0}\int_{\Omega_{T}}\partial_{t}h(u^{q})\frac{\,u\,}{\varepsilon} \chi_{\{a \leq u \leq a+\varepsilon\}}\,\varphi\,dz =a \int _{\Omega_T \cap\{u = a\}}\partial_t h (u^q)\varphi\,dz.
 \end{equation}
 By Lebesgue's dominated convergence theorem, the left hand side of (\ref{eq.A1}) is computed as
 \begin{equation}\label{eq.A4}
\int_{\Omega_{T}}u \phi_{\varepsilon}(u)\partial_{t}\varphi\,dz \to \int_{\Omega_{T}}u \chi_{\{u\geq a\}}\partial_{t}\varphi\,dz.
 \end{equation}
 Passing to the limit as $\varepsilon \searrow 0$  in (\ref{eq.A1}) and gathering (\ref{eq.A2}), (\ref{eq.A3}) and (\ref{eq.A4}), we have
\begin{equation*} 
 \int_{\Omega_{T}}u \chi_{\{u \geq a\}}\partial_{t}\varphi\,dz=-\int_{\Omega_{T}}\partial_{t}h(u^{q}) \chi_{\{u \geq a\}} \varphi\,dz+a \int _{\Omega_T \cap\{u = a\}}\partial_t h (u^q)\varphi\,dz.
 \end{equation*}
 Again, by Lebesgue's dominated theorem, taking the limit as $a \searrow 0$ in the above formula, we have
 \begin{equation*}
 \int_{\Omega_{T}}u\partial_{t}\varphi\,dz=-\int_{\Omega_{T}}\partial_{t}h(u^{q})\varphi\,dz=-\int_{\Omega_{T}}\partial_{t}u\cdot \varphi\,dz,
 \end{equation*}
 which completes the proof.
 \end{proof}

In the proof above, we used the following lemma as for the convergence of Dirac measure.

\begin{lem}\label{conv. Dirac}
Let $a>0$ and $\varepsilon >0$ be arbitrary given. Then \begin{equation*}
\frac{1}{\varepsilon}\chi_{\{a\leq u \leq a+\varepsilon\}} \to \delta_{(a)}\quad as \,\,\varepsilon \searrow 0\quad \textrm{in}\,\,\mathscr{D}'(\mathbb{R}),
\end{equation*}where we denote by $\mathscr{D}'(\mathbb{R})$ the distribution function space, which is the dual space of the space $\mathscr{D}$ of smooth functions with compact support in $\mathbb{R}$.
\end{lem}
\begin{proof}
For any $\psi \in \mathscr{D}$, we have
\begin{align*}
\int_{-\infty}^{\infty}\frac{1}{\varepsilon}\chi_{\{a\leq u \leq a+\varepsilon\}}\psi \,du=\frac{1}{\varepsilon} \int_{a}^{a+\varepsilon}\psi du \to \psi(a)=\langle \delta_{(a)},\,\psi\rangle
\end{align*}
as $\varepsilon \searrow 0$. Therefore we have for any $\psi \in \mathscr{D}$ \begin{equation*}
\lim_{\varepsilon \searrow 0} \bigg\langle \frac{1}{\varepsilon}\chi_{\{a\leq u \leq a+\varepsilon\}},\,\psi \bigg\rangle =\langle \delta_{(a)},\,\psi \rangle,
\end{equation*}
which is our claim.
\end{proof}

 \subsection{Regularization}
 In this subsection, we will show the following regularization, Lemma \ref{regularization}. Before stating assertion, we prepare some notations. For $f \in L^{1}_{loc}(\Omega_{T})$, we denote the \textit{mollifier} of $f$ by \begin{equation*}
 f_{\varepsilon,h}(z):=\int_{Q_{\varepsilon,h}(z)}\rho_{\varepsilon,h}(z^\prime-z)f(z^\prime)\,dz^\prime.
 \end{equation*}Here $\varepsilon > 0$, $z=(x,t)$, $Q_{\varepsilon,h} (z)=B_\varepsilon (x)\times (t-h, t+h)\subset \Omega_T$, and
 \begin{equation*}
 \rho_{\varepsilon,h}(z):=\frac{1}{h}\rho_{1}\left(\frac{t}{h}\right)\frac{1}{\varepsilon^{n}}\rho_{2}\left(\frac{x}{\varepsilon} \right),
 \end{equation*}
 where $\rho_1$ and $\rho_2$ are smooth symmetric in the following sense:
\begin{equation*}
\rho_1(t)=\rho_1(-t), \quad \rho_2 (x)= \rho_2 (|x|),
\end{equation*}
and satisfies
 \begin{equation*}
 \supp (\rho_1) \subset (-1,1),\,\,
\int_{\mathbb{R}} \rho_1(t)dt=1\quad ; \quad
\supp (\rho_2) \subset B_1 (0),\,\,
\int_{\mathbb{R}^n} \rho_2 (x) d x=1.
 \end{equation*}
\begin{lem}\label{regularization}
 Let $0<t_1<t <T$. For the weak solution $u$ to (\ref{maineq'}),
 \begin{equation*}
 \int_{\Omega_{t_1,t}}|\nabla u|^{p-2}\nabla u\cdot \nabla \left((\partial_{t}u)_{\varepsilon,h}\right)_{\varepsilon,-h}\,dz \longrightarrow \int_\Omega
\frac{\,1\,}{p} |\nabla u|^p\,dx\Bigg|^{t}_{t_1}\quad as \quad \varepsilon, h \searrow 0.
 \end{equation*}
 \end{lem}

 \begin{proof}
 From (D1) in Definition \ref{def of weak sol.} and Fubini's theorem, it follows that
 \begin{align}\label{conv. eq0}
&\int_{\Omega_{t_1,t}}|\nabla u|^{p-2}\nabla u \cdot \nabla \left((\partial_{t}u)_{\varepsilon, h}\right)_{\varepsilon, -h}\,dz \notag \\
&=\int_{\Omega_{t_1,t}}\left(|\nabla u|^{p-2}\nabla u \right)_{\varepsilon, h} \cdot \nabla (\partial_{t}u)_{\varepsilon, h}\,dz \notag \\
 &\quad =-\int_{\Omega_{t_1,t}}\bigg[ \mathrm{div} \left(|\nabla u|^{p-2}\nabla u \right)_{\varepsilon, h} -\mathrm{div}\left(|\nabla u_{\varepsilon, h}|^{p-2}\nabla u_{\varepsilon, h} \right) \bigg]\partial_{t}u_{\varepsilon, h}\,dz \notag\\
 &\quad \quad \quad \quad \quad +\int_{\Omega_{t_1,t}}|\nabla u_{\varepsilon, h}|^{p-2}\nabla u_{\varepsilon, h} \cdot  \partial_{t}(\nabla u_{\varepsilon, h})\,dz \notag \\
 &\quad =:I+J.
 \end{align}
 Since $\mathrm{div} (|\nabla u|^{p-2}\nabla u) \in L^{2}(\Omega_{T})$ by (D1) and (D2) again, we have, as $\varepsilon, h \searrow 0$,
\begin{equation}\label{conv. eq1}
 \Big(\mathrm{div} \left(|\nabla u|^{p-2}\nabla u \right)\Big)_{\varepsilon, h},\,\, \mathrm{div} (|\nabla u_{\varepsilon, h}|^{p-2} \nabla u_{\varepsilon, h})  \to \mathrm{div}\left(|\nabla u|^{p-2}\nabla u \right)\quad \textrm{strongly}\,\,\textrm{in}\,\,L^{2}(\Omega_{T}),
 \end{equation}
 and, by Lemma \ref{existence of u_t}, as $\varepsilon, h \searrow 0$,
 \begin{equation}\label{conv. eq2}
 (\partial_{t}u)_{\varepsilon, h} \to \partial_{t}u\quad \textrm{strongly}\,\,\textrm{in}\,\,L^{2}(\Omega_{T}).
 \end{equation}
 It follows from (\ref{conv. eq1}) and (\ref{conv. eq2}) that
 \begin{equation}\label{conv. eq3}
 I \to 0\quad \textrm{as} \quad \varepsilon, h \searrow 0.
 \end{equation}
 By (D1) in Definition \ref{def of weak sol.} we have, as $\varepsilon, h \to 0$,
 \begin{align}\label{conv. eq4}
J=\int_\Omega \frac{1}{p}|\nabla u_{\varepsilon, h}|^p
d x \Bigg|^{t}_{t_1}
\to \int_\Omega
 \frac{1}{p}|\nabla u|^p
d x \Bigg|^{t}_{t_1}
 \end{align}
 and thus, gather (\ref{conv. eq3}), (\ref{conv. eq4}) and (\ref{conv. eq0}) to complete the proof.
 \end{proof}

\section{Proof of Proposition \ref{Energy equality pS}}\label{Sec. Proof energy equality}

This section is devoted to prove Proposition \ref{Energy equality pS}.
 \begin{proof}[Proof of Proposition \ref{Energy equality pS}]
 (i)\quad Let $0< t_1 < t < T$ be arbitrarily taken and let $\sigma_{t_1,t}$ be the same time cut-off function as in the proof of Proposition \ref{Nonnegativity prop}. The function $\sigma_{t_1, t}u$ is in $L^\infty (t_1, t ; W^{1, p} (\Omega))$, and nonnegative by Proposition \ref{nonnegativity pS}
 and thus, is an admissible test function in (D2) for \eqref{pS}.
Choose a test function as $\sigma_{t_1, t} u$ in (D2) for \eqref{pS}, to have
\begin{equation}\label{energy ineq1}
\int_{\Omega_{t_{1},t}}\partial_{\tau}(u^{q})\sigma_{t_1, t}u\,dz+\int_{\Omega_{t_{1},t}}|\nabla u|^{p-2}\nabla u\cdot \nabla\left(\sigma_{t_1, t}u\right)\,dz = \int_{\Omega_{t_{1},t}}\lambda(\tau)u^{q}\sigma_{t_1, t}u\,dz.
\end{equation}
By the very definition of $A^+(u)$ for $u \geq 0$, the first term on the left hand side of (\ref{energy ineq1}) is computed as
\begin{align}\label{energy ineq2}
\int_{\Omega_{t_{1},t}}\partial_{\tau}(u^{q})\sigma_{t_1, t}u\,dz&=\int_{\Omega_{t_{1},t}}\partial_{\tau}A^+(u)\sigma_{t_1, t}\,dz \notag \\
&=\int_{\Omega} A^+(u)\sigma_{t_1, t}\,dx \bigg|_{t_1}^t-\int_{\Omega_{t_{1},t}}A^+(u)\partial_\tau\sigma_{t_1, t}\,dz \notag \\
&\to\int_\Omega \frac{q}{q+1}u(t)^{q+1}\,dx-\int_\Omega \frac{q}{q+1}u(t_1)^{q+1}\,dx \quad \textrm{as}\quad \delta \searrow 0.
\end{align}
The second term on the left hand side of (\ref{energy ineq1}) is treated as
\begin{align}\label{energy ineq3}
\int_{\Omega_{t_{1},t}}|\nabla u|^{p-2}\nabla u\cdot \nabla \left(\sigma_{t_1, t}u\right)\,dz = \int_{\Omega_{t_{1},t}}|\nabla u|^{p}\sigma_{t_1, t} dz \to \int_{\Omega_{t_{1},t}}|\nabla u|^{p}\,dz \quad \textrm{as}\quad \delta \searrow 0.
\end{align}
Using (\ref{energy ineq2}), (\ref{energy ineq3}) and the volume preserving condition $\displaystyle \int_{\Omega}u(x,t)^{q+1}\,dx=1,\,\,t \geq 0$, we take the limit as $\delta \searrow 0$ in (\ref{energy ineq1}) to obtain that
\begin{equation*}
\int_{\Omega_{t_{1},t}}|\nabla u|^{p}\,dz=\int_{\Omega_{t_{1},t}}\lambda(\tau)u^{q+1}\,dz.\notag
\end{equation*}
Dividing above formula by $t-t_{1}$, we have
\begin{equation*}
\frac{1}{t-t_{1}}\int_{t_{1}}^{t}\int_{\Omega}|\nabla u(x,\tau)|^{p}dxd\tau=\frac{1}{t-t_{1}}\int_{t_{1}}^{t} \int_{\Omega}\lambda(\tau)u^{q+1}(x,\tau)\,dxd\tau.
\end{equation*}
According to the volume preserving condition again, passing the limit as  $t \searrow t_{1}$ in the formula above, we obtain that
\begin{equation*}
\lambda(t_{1})=\int_{\Omega}|\nabla u(x,t_{1})|^{p}\,dx,
\end{equation*}
which is our first assertion.
%
%

%
%
%
\bigskip
\noindent
\\
\noindent
(ii)\quad We notice the boundedness of the solution $u$ of the p-Sobolev flow. This is shown as follows:
By Proposition \ref{Energy equality pS} (i) above,
$\lambda (t) = \|\nabla u (t)\|^{p}_{L^p (\Omega)}$ and thus,
$\lambda (t) \in L^\infty (0, T)$
by (D1) in Definition \ref{def of weak sol.}.
We also have Proposition \ref{Energy equality pS} (i) that $(u)_+$ is bounded in $\Omega_T$ as in Proposition \ref{Boundedness of the p-Sobolev flow}, and thus, $u$ itself bounded by Proposition \ref{Nonnegativity prop pS}.
Consequently, the function $\sigma_{t_1, t}\partial_tu$ is an admissible test function in (D2) of Definition \ref{def of weak sol.}
by Lemmata \ref{existence of u_t} and \ref{regularization}.
We now take a test function as $\sigma_{t_1, t} \partial_tu$ in (D2) of Definition \ref{def of weak sol.} and then

\begin{equation}\label{energy ineqb1}
\int_{\Omega_{t_{1},t}}\partial_{t}(u^q)\sigma_{t_1, t}\partial_tu\,dz+\int_{\Omega_{t_{1},t}}|\nabla u|^{p-2}\nabla u\cdot \nabla \left(\sigma_{t_1, t}\partial_tu\right)\,dz = \int_{\Omega_{t_{1},t}}\lambda(t)u^q\sigma_{t_1, t}\partial_tu\,dz,
\end{equation}
Note that the integral on the right hand side in (\ref{energy ineqb1}) is finite by Proposition \ref{Boundedness} and Lemma \ref{existence of u_t}.
Using the Lebesgue dominated theorem with Proposition \ref{Boundedness} and Lemma \ref{existence of u_t}, the first term on the left hand side of (\ref{energy ineqb1}) is computed as
\begin{align}\label{energy ineqb2}
\int_{\Omega_{t_{1},t}}\partial_{t}u^q\sigma_{t_1, t}\partial_tu\,dz
&=q\int_{\Omega_{t_1,t}}u^{q-1}(\partial_t u)^2\,\sigma_{t_1, t}dz \notag \\
&\to q\int_{\Omega_{t_1,t}}u^{q-1}(\partial_t u)^2\,dz\quad \textrm{as}\quad \delta \searrow 0.
\end{align}
%
%
%
%
The second term on the left hand side of (\ref{energy ineqb1}) is  treated as
\begin{align}\label{energy ineqb3}
\int_{\Omega_{t_1,t}}|\nabla u|^{p-2}\nabla u\cdot \nabla \left(\sigma_{t_1, t}\partial_tu\right)\,dz&=\int_{\Omega_{t_1,t}}|\nabla u|^{p-2}\nabla u\cdot \partial_t\nabla u\sigma_{t_1, t}\,dz \notag \\
&=\int_{\Omega_{t_1,t}} \partial_t \left(\frac{\,1\,}{p}|\nabla u|^p\right)\sigma_{t_1, t}\,dz \notag \\
&=\int_\Omega \frac{\,1\,}{p}|\nabla u|^p\,\sigma_{t_1, t}\,dx\Bigg|_{t_1}^t-\int_{\Omega_{t_1,t}}  \frac{\,1\,}{p}|\nabla u|^p\partial_t\sigma_{t_1, t}\,dz \notag \\
&\to \int_\Omega \frac{\,1\,}{p}|\nabla u(x,t)|^p\,dx-\int_\Omega \frac{\,1\,}{p}|\nabla u(x,t_1)|^p\,dx\quad \textrm{as}\quad \delta \searrow 0 \notag \\
&=\frac{\,1\,}{p}\lambda(t)-\frac{\,1\,}{p}\lambda(t_1),
\end{align}
where the manipulation in the second and third lines are justified by Lemma \ref{regularization} in Appendix \ref{Sec. Some fundamental facts}. 
By the volume conservation $\int_\Omega u(x,t)^{q+1}=1, \,\,t \geq 0$, the right hand side of (\ref{energy ineqb1}) is calculated as
\begin{align}\label{energy ineqb4}
\int_{\Omega_{t_1,t}}\lambda(t)u^q\sigma_{t_1, t}\partial_tu\,dz
&=\int_{t_1}^t \lambda(t) \sigma_{t_1,t} \frac{d}{dt}\left(\int_\Omega \frac{u^{q+1}}{q+1}\,dx \right)\,dt 
=0.
\end{align}
From (\ref{energy ineqb2}), (\ref{energy ineqb3}) and (\ref{energy ineqb4}), it follows that
\begin{equation*}
q\int_{\Omega_{t_1,t}}u^{q-1}(\partial_t u)^2\,dz+\frac{\,1\,}{p}\lambda(t)-\frac{\,1\,}{p}\lambda(t_1)=0.
\end{equation*}
Letting $t_1=0$, we have the desired result.

\end{proof}

 \section{Notes on H\"{o}lder regularity}\label{Sec. Notes on Holder regularity}

 \subsection{A local energy estimate for (\ref{eq. of v})}\label{subsect. energy of v}
 We will derive a local energy estimate for (\ref{eq. of v}) here.

 \begin{lem}\label{energy of v}
Let $\theta>0$ be a parameter.  For any $z_0=(x_0,t_0) \in \Omega'_T$, take $\rho>0$ such that  $Q(\theta,\rho)(z_0)\equiv B_\rho(x_0) \times (t_0-\rho^\theta,t_0) \subset \Omega'_T$. Let $\zeta$ be a piecewise smooth function on $Q(\theta,\rho)(z_0)$ satisfying
\begin{equation*}
0 \leq \zeta \leq 1,\quad |\nabla \zeta|<\infty,\quad \zeta(x,t)=0\quad \textrm{outside}\,\,Q(\theta,\rho)(z_0).
\end{equation*}
Furthermore, take a positive number $\delta_0$ such that $$\displaystyle \esssup_{Q(\theta,\rho)(z_0)}|(k-v)_+|,\,\,\esssup_{Q(\theta,\rho)(z_0)}|(v-k)_+|\leq \delta_0$$ for some $k\geq 0$. Then the following inequality holds true.
\begin{enumerate}[(i)]
\item Let $v$ be a weak supersolution to (\ref{eq. of v}). Then it holds that
\begin{align}\label{local energy ineq of v}
&\esssup_{t_0-\rho^\theta <t<t_0}\int_{B_\rho(x_0)}(k-v)_+^2\zeta^p\,dx+\int_{Q(\theta,\rho)(z_0)}|\nabla (k-v)_+|^p\zeta^p\,dz \notag \\
&\quad \leq \int_{B_\rho(x_0)\times \{t_0-\rho^\theta\}}(k-v)_+^2\zeta^p\,dx+C\int_{Q(\theta,\rho)(z_0)}(k-v)_+^p|\nabla \zeta|^p\,dz \notag \\
&\qquad \qquad +\int_{Q(\theta,\rho)(z_0)}(k-v)_+^2\zeta^{p-1}|\zeta_t|\,dz+C\delta_0 \int_{Q(\theta,\rho)(z_0)}\chi_{\{(k-v)_+>0\}}dz,
\end{align}
where $C$ is a positive constant depending only on $n,p,\tilde{c},M,\lambda(0)$.

\item Let $v$ be a weak subsolution to (\ref{eq. of v}). Then it holds that
\begin{align}\label{local energy ineq of v}
&\esssup_{t_0-\rho^\theta <t<t_0}\int_{B_\rho(x_0)}(v-k)_+^2\zeta^p\,dx+\int_{Q(\theta,\rho)(z_0)}|\nabla (v-k)_+|^p\zeta^p\,dz \notag \\
&\quad \leq \int_{B_\rho(x_0)\times \{t_0-\rho^\theta\}}(v-k)_+^2\zeta^p\,dx+C\int_{Q(\theta,\rho)(z_0)}(v-k)_+^p|\nabla \zeta|^p\,dz \notag \\
&\qquad \qquad +\int_{Q(\theta,\rho)(z_0)}(v-k)_+^2\zeta^{p-1}|\zeta_t|\,dz+C\delta_0 \int_{Q(\theta,\rho)(z_0)}\chi_{\{(v-k)_+>0\}}dz,
\end{align}
where $C$ is a positive constant depending only on $p,n,\tilde{c},M, \lambda(0)$.
\end{enumerate}
\end{lem}

\begin{proof} We give the proof only for the case (i).
Take a test function as $\varphi=-(k-v)_+\zeta^p$ in the weak form of (\ref{eq. of v}); i.e.,
\begin{equation} \label{weak form eq. of v}
-\int_{Q_t(\theta, \rho)(z_0)}v\partial_t\varphi dz+\int_{Q_t(\theta, \rho)(z_0)}|\nabla v|^{p-2} \nabla v\cdot \nabla\varphi\,dz=c\int_{Q_t(\theta, \rho)(z_0)}v\varphi\,dz,
\end{equation}
where $Q_t(\theta, \rho)(z_0):=B_\rho(x_0) \times (t_0-\rho^\theta, t)$ for any $t \in (t_0-\rho^\theta, t_0) $. The first term on the left hand side of (\ref{weak form eq. of v}) is computed as
\begin{equation}\label{energy ineq1 of v}
\int_{B_\rho(x_0)\times \{t\}}\frac{1}{2}(k-v)_+^2\zeta^p\,dx \bigg|_{t_0-\theta}^t-\int_{Q_t(\theta, \rho)}\frac{1}{2}(k-v)_+^2p\zeta^{p-1}\zeta_t\,dz .
\end{equation}
Meanwhile, by use of (\ref{bounded from above and below1}) and Young's inequality, the second term of (\ref{weak form eq. of v}) is estimated from below as
\begin{align}\label{energy ineq2 of v}
&\int_{Q_t(\theta, \rho)(z_0)}|\nabla v|^{p-2}\nabla v\cdot \nabla (-(k-v)_+\zeta^p)\,dz \notag \\
&=\int_{Q_t(\theta, \rho)(z_0)}|\nabla(k-v)_+|^{p-2} \nabla (k-v)_+\cdot \nabla(k-v)_+\zeta^p\,dz \notag \\
&\qquad \qquad +\int_{Q_t(\theta, \rho)(z_0)}|\nabla (k-v)_+|^{p-2} \nabla (k-v)_+\cdot \big((k-v)_+p\zeta^{p-1}\nabla\zeta\big)\,dz \notag \\
&\geq c_0\int_{Q_t(\theta, \rho)(z_0)}|\nabla (k-v)_+|^{p}\zeta^p\,dz-c_1\int_{Q_t(\theta, \rho)(z_0)}|k-v)_+^{p}|\nabla \zeta|^p\,dz,
\end{align}
where $c_0$ and $c_1$ are positive constants depending only on $p,\,n,\,M$. By using \eqref{lambdaineq} in Proposition \ref{Energy equality pS} the right hand side of (\ref{weak form eq. of v}) is bounded above by
\begin{equation}\label{energy ineq3 of v}
\bigg| \lambda(t)\int_{Q_t(\theta, \rho)(z_0)}v(-(k-v)_+\zeta^p)\,dz \bigg| \leq c_2 \delta_0\int_{Q_t(\theta, \rho)(z_0)}\chi_{\{(k-v)_+>0\}}dz,
\end{equation}
where $c_2$ is a positive constant depending only on $\tilde{c}, \lambda(0)$.
Gathering (\ref{energy ineq1 of v}), (\ref{energy ineq2 of v}) and (\ref{energy ineq3 of v}), we arrive at the desired estimate (\ref{local energy ineq of v}). \end{proof}

\subsection{Outline of proof of Theorem \ref{Gradient Holder continuity}}\label{subsect. Outline of proof of Theorem Gradient Holder continuity}

We recall the outline of proof of Theorem \ref{Gradient Holder continuity} here.
\begin{proof}
By the H\"older continuity in Theorem \ref{Holder continuity}, the equation (\ref{eq. of v}) is an evolutionary $p$-Laplacian system
with H\"{o}lder continuous elliptic and boudedness coefficients $g$ and lower order terms $v$.
We apply the gradient H\"{o}lder regularity for the evolutionary $p$-Laplacian systems with lower order terms in
\cite[Theorem 1, p.390]{Misawa}
(also see \cite{Karim-Misawa}).
Here the so-called Campanatto's perturbation method is applied to the gradient H\"{o}lder regularity for
the evolutionary $p$-Laplacian systems with H\"older coefficients and lower order terms.
We also refer to the book in \cite[Theorem 1.1, p.245]{DiBenedetto1}.
\end{proof}

\end{appendices}

\end{document}